\RequirePackage{fix-cm}
\documentclass{svjour3}                     
\smartqed  

\usepackage{dsfont}
\usepackage{amssymb}
\usepackage{graphicx}
\usepackage{xcolor}
\usepackage{graphicx}
\usepackage{amsmath}
\usepackage{epstopdf} 
\usepackage{mathptmx}      

\begin{document}

\title{A new recursive spectral Tau method on system of generalized Abel-Volterra integral equations
}

\author{S. Shahmorad \and P. Mokhtary \and Y. Talaei  \and A. Faghih
}

\authorrunning{ S. Shahmorad , P. Mokhtary \and Y. Talaei,} 

\institute{ Sedaghat Shahmorad (Corresponding author)\at
             Department of Applied Mathematics, Faculty of Mathematical Science, University of Tabriz, Tabriz, Iran.\\
               \email{shahmorad@tabrizu.ac.ir}
                \and
             Payam Mokhtary \at
             Department of Mathematics, Faculty of Basic Sciences, Sahand University of Technology, Tabriz, Iran.\\
              \email{mokhtary.payam@gmail.com,\; mokhtary@sut.ac.ir}
              \and 
         Younes Talaei \at
              Department of Applied Mathematics, Faculty of Mathematical Science, University of Tabriz, Tabriz, Iran. \\
              \email{y$\_$talaei@tabrizu.ac.ir}           
           \and      
         Amin Faghih \at
               Department of Mathematics, Faculty of Basic Sciences, Sahand University of Technology, Tabriz, Iran. \\
              \email{amin.fagheh71@gmail.com}           
}

\maketitle

\begin{abstract}
This paper provides an efficient recursive approach of the spectral Tau method, to approximate the solution of system of generalized Abel-Volterra integral equations. In this regards, we first investigate the existence, uniqueness as well as smoothness of the solutions under various assumptions on the given data. Next, from a numerical perspective, we express approximated solution as a linear combination of suitable canonical polynomials which are constructed by an easy to use recursive formula. Mostly, the unknown parameters are calculated by solving a low dimensional algebraic systems independent of degree of approximation which prevent from high computational costs. Obviously, due to singular behavior of the exact solution,  using classical polynomials to construct canonical polynomials, leads to low accuracy results. In this regards, we develop a new fractional order canonical polynomials using M\"untz-Legendre polynomials which have a same asymptotic behavior with the solution of underlying problem. The convergence analysis is discussed, and the familiar spectral accuracy is achieved  in $L^{\infty}$-norm. Finally, the reliability of the method is evaluated using various problems.

\end{abstract}
\section{Introduction}
In this paper, the following system of generalized Abel-Volterra integral equations is considered:
\begin{equation}\label{1}
y_i(t)=g_i(t)+\sum\limits_{j=1}^{n}{\int\limits_{0}^{t}{(t-s)^{\alpha_{ij}-1}k_{ij}(t,s)y_j(s)ds}},\quad i=1,2,...,n,~~t\in \Omega=[0,1],
\end{equation}
where $k_{ij}(t,s)$, and $g_{i}(t)$ are given continuous functions for $i,j=1,2,...,n$. $y_i(t)$ is unknown and
\begin{equation}\label{129}
\alpha_{ij}=\frac{a_{ij}}{b_{ij}} \in \mathbb{Q}\cap (0,1],\ \ \
gcd(a_{ij},b_{ij})=1,\ \ i,j=1,2,...,n.
\end{equation}

These type of integral equations  arise in many scientific applications such as diffusion problems, spread of epidemics, the behavior of viscoelastic materials
in mechanics (see \cite{Brau,Lighthill,Abel} and references therein). There are several numerical methods for solving Abel-Volterra integral equations system such as waveform
relaxation method \cite{H4}, extrapolation method \cite{Tao}, block by block method \cite{Katani}, reproducing kernel Hilbert space
method \cite{Ha}, parallel algorithm \cite{H6}, Laplace transform method \cite{Shamloo}. However,  few research has been done on the Abel-Volterra integral equations system (\ref{1}).

The recursive approach of Tau method was firstly introduced in 1969 by Ortiz  \cite{Ortiz1} to solve a class of ordinary differential equations, and was extended for solving a system of ODE's in \cite{Freilich,Russo}. This method firstly considered the approximate solution as a linear
combination of some suitable functions called canonical polynomials which are calculated recursively, and secondly transformed
the underlying problem to an equivalent system of algebraic equations. This approach have two main advantages.  First, to produce a more reliable approximate solution,
it is not necessary to repeat whole computation, since in the next step, the already determined canonical polynomials are used again, and second, the dimension of the resulting algebraic system is independent of the degree of approximation and equal to the height of operator.
Clearly, this property can prevent unwanted oscillations in errors that may occur
for large degree of approximations.

In functional equations with smooth solution, the classical orthogonal polynomials such as Jacobi polynomials are used in constructing the canonical polynomials. Since we have weakly singular integral operators, we typically expect that the solutions of \eqref{1} are not smooth, even for smooth input functions. Evidently, this can effect constructing the numerical methods with a reasonable accuracy. This drawback motivates us to design a powerful approach to approximate the non-smooth solutions of \eqref{1}. In this regards, recognizing the smoothness properties of the solutions and designing approximate approach that  produces the numerical solutions with a same asymptotic behavior with the exact ones are essential. In this paper, we first investigate existence, uniqueness and smoothness properties of \eqref{1}. We prove that the first derivatives of the solutions of \eqref{1} suffer from a discontinuity at origin. Consequently, developing the classical approach for the numerical  solution of  \eqref{1} yields a low accurate approximations, and thereby producing high order approximate solutions regardless of singularity behavior of the solutions is crucial and new in the literature. Then we introduce a new vector of fractional canonical polynomials with a same asymptotic behavior with the solutions of \eqref{1} using suitable  M\"untz-Legendre polynomials, and develop a high order recursive Tau approach.

The organization of the paper is as follows: In the next section, we give some required definitions and preliminaries. In Section \ref{sec1}, existence, uniqueness, and smoothness properties of \eqref{1} are investigated. Section  \ref{sec2} devoted to our numerical approach. In this section, we begin by some definitions and theorems which are required in the sequel, and then construct a recursive Tau approach based on newly defined fractional vector canonical polynomials to approximate the solution of \eqref{1}. The convergence analysis of the proposed method is investigated in Section \ref{sec3}. Section \ref{sec4} is devoted  to numerical results to illustrate the effectiveness of the proposed method. Finally, some conclusion is presented in Section \ref{sec5}.
\section{Preliminaries}
In this section, some required definitions and lemmas in the sequel, are presented.
\subsection{Shifted Jacobi polynomials}
The shifted Jacobi polynomials on $\Omega$ denoted by $J_n^{\vartheta,\xi}(s)$ are orthogonal with respect to the shifted Jacobi weight function $w^{\vartheta,\xi}(s)=s^{\xi}(1-s)^{\vartheta}$ with the parameters $\vartheta,\xi>-1$, i.e.,
\begin{equation}\label{erev1}
\int_{\Omega}J_{m}^{\vartheta,\xi}(s)J_{n}^{\vartheta,\xi}(s)w^{\vartheta,\xi}(s)ds=h_{n}^{\vartheta,\xi}\delta_{mn}, \quad  m,n\geq 0,
\end{equation}
in which
\begin{equation*}
h_{n}^{\vartheta,\xi}=\|J_{n}^{\vartheta,\xi}\|_{w^{\vartheta,\xi}}^{2}=\dfrac{\Gamma(n+\vartheta+1)\Gamma(n+\xi+1)}{(2n+\vartheta+\xi+1)n!\Gamma(n+\vartheta+\xi+1)},
\end{equation*}
and $\delta_{mn}$ is the Kronecker delta, and $\|.\|_{\vartheta,\xi}$ is the shifted Jacobi $L^2$-norm defined by
\[\|U\|_{w^{\vartheta,\xi}}^2=(U,U)_{w^{\vartheta,\xi}}:=\int\limits_{\Omega}{U^2(s)w^{\vartheta,\xi}(s)ds}.\]

The shifted Jacobi $L^2$ space on $\Omega$ is defined by
\[
L_{w^{\vartheta,\xi}}^2(\Omega)=\{U:\Omega \to \mathbb{R}, \|U\|_{w^{\vartheta,\xi}} < \infty\}.
\]

For simplicity we use the notations $(.,.), (L^2(\Omega),\|.\|)$, when
$\vartheta=\xi=0$. These polynomials have the following explicit formula
\begin{equation*}
J_{n}^{\vartheta,\xi}(s)=\sum_{j=0}^{n}Z_{j}^{\vartheta,\xi,n}s^{j},
\end{equation*}
where
\begin{equation*}
  Z_{j}^{\vartheta,\xi,n}=\dfrac{(-1)^{n-j}\Gamma(n+\xi+1)\Gamma(n+\vartheta+\xi+j+1)}{\Gamma(\xi+j+1)j!\Gamma(n+\vartheta+\xi+1)(n-j)!}.
\end{equation*}

Further properties of the Jacobi polynomials can be found in \cite{Shens}.
The shifted Jacobi orthogonal projection $\Pi_N^{\vartheta,\xi}:L_{w^{\vartheta,\xi}}^2(\Omega) \to \mathbb{P}_N$ is defined by
\begin{equation}\label{eqq3}
\Pi_N^{\vartheta,\xi}U(s)=\sum_{j=0}^{N}U_j^{\vartheta,\xi} J_j^{\vartheta\xi}(s),
\end{equation}
where $\mathbb{P}_{N}$ is the space of all algebraic polynomials with degree at most $N$, and
\begin{equation*}
U_j^{\vartheta,\xi}=\dfrac{1}{h_{j}^{\vartheta,\xi}} (U, J_j^{\vartheta,\xi})_{w^{\vartheta,\xi}}.
\end{equation*}

The following lemma gives the bound of the truncation error $U-\Pi_N^{\vartheta,\xi}U$ for various values of $\vartheta, \xi$ in the $L_\infty$ norm.
\begin{lemma}\label{lemrev1}
Let $\partial_s^k U \in L_{w^{\vartheta,\xi}}^{2}(\Omega)$ for $k \geq 1$. Then there exists a positive constant $C$ independent of $N$ such that
\begin{equation}\label{erev3}
\|U-\Pi_N^{\vartheta,\xi}U\|_{\infty} \leq C \begin{cases} N^{\frac{3}{4}-k}\|\partial_s^k U\|_{w^{\vartheta,\xi}},\hspace{2.5 cm} -1< \vartheta,\xi \le 0,\\
\\
(1+N^{\rho+\frac{1}{2}})N^{\frac{3}{4}-k}\|\partial_s^kU\|_{w^{\vartheta,\xi}},\hspace{1 cm} \varrho=\max(\vartheta,\xi),~ \text{Else}.
\end{cases}
\end{equation}
\end{lemma}
\begin{proof}
The first inequality is given in \cite{rrev1}. To prove the other ones, we can write
\begin{equation*}
\begin{split}
\|U-\Pi_{N}^{\vartheta,\xi}U\|_{\infty}&=\|\big(U-\Pi_{N}^{-\frac{1}{2},-\frac{1}{2}}U\big)-\Pi_{N}^{\vartheta,\xi}\Big(U-\Pi_{N}^{-\frac{1}{2},-\frac{1}{2}}U\Big)\|_{\infty}\\
&\leq (1+\|\Pi_{N}^{\vartheta,\xi}\|_{\infty})\|U-\Pi_{N}^{-\frac{1}{2},-\frac{1}{2}}U\|_\infty.
\end{split}
\end{equation*}

From \cite{mokhs3}, we have
\begin{equation}\label{erev2}
\|\Pi_{N}^{\vartheta,\xi}\|_{\infty}=\mathcal{O}\left(N^{\varrho+\frac{1}{2}}\right),\quad \varrho=\max(\vartheta,\xi).
\end{equation}

Thus, the desired inequality can be obtained by applying \eqref{erev2} and the first inequality of \eqref{erev3}.
\end{proof}
\subsection{M\"untz-Legendre polynomials}
Setting $\mathrm{V}=\{v_i:=i \sigma,~~\sigma \in (0,1)\}_{i=0}^{\infty}$, the M\"{u}ntz space
associated with $\mathrm{V}$ is defined by
\[\mathrm{M}_{N,\sigma}(\mathrm{V})=\text{Span}\{
1, t^{\sigma},..., t^{N \sigma}\},\quad
\mathrm{M}(\mathrm{V})=\bigcup\limits_{N=0}^{\infty}
\mathrm{M}_{N,\sigma}(\mathrm{V}).\]

The M\"{u}ntz-Jacobi polynomial with index $(0,\frac{1}{\sigma}-1)$ is
defined by
\begin{equation*}
L_{i,\sigma}(t):=\mathrm{J}^{0,\frac{1}{\sigma}-1}_{i}(t^{\sigma}) \in
\mathrm{M}_{i \sigma,\sigma}(\mathrm{V}) ,~~~\forall~ i \in \mathbb{N}_0,~ t \in \Omega.
\end{equation*}

In view of \eqref{erev1}, it can be easily checked that these functions are mutually orthogonal, i.e.,
\begin{equation}\label{erev4}
\begin{split}
\int_{\Omega}L_{m,\sigma}(t)L_{n,\sigma}(t)dt=\frac{1}{\sigma}
h_{n}^{0,\frac{1}{\sigma}-1}\delta_{mn}, \quad  m,n\geq 0.
\end{split}
\end{equation}

These polynomials satisfy the following recurrence relation \cite{Calcolo}
\[
\left\{\begin{array}{l}
L_{0,\sigma}(t)=1,\ \ \ L_{1,\sigma}(t)=\frac{t^\sigma(1+\sigma)-1}{\sigma},\\
\\
d_{1,i}L_{i+1,\sigma}(t)=d_{2,i}(t)L_{i,\sigma}(t)-d_{3,i}L_{i-1,\sigma}(t),\ \ \ i=1,2,...,
\end{array}\right.
\]
where
\[
\left\{\begin{array}{l}
d_{1,i}=2(i+1)(i+\frac{1}{\sigma})(2i+\frac{1}{\sigma}-1), \\
\\
d_{2,i}(t)=(2i+\frac{1}{\sigma})\Big((2i+\frac{1}{\sigma}-1)(2i+\frac{1}{\sigma}+1)t-(\frac{1}{\sigma}-1)^{2}\Big), \\
\\
d_{3,i}=2i(i+\frac{1}{\sigma}-1)(2i+\frac{1}{\sigma}+1).
\end{array}\right.
\]

The M\"{u}ntz-Jacobi polynomials $\{L_{i,\sigma}(t)\}_{i \ge 0}$ are
mutually orthogonal and form a complete orthogonal system in
$L^2(\Lambda)$ \cite{mokhs,Shenm}. That is, for any $u(t)\in L^2(\Omega)$, we have the
following unique expansion
\begin{equation*}\label{eqrev4}
u(t)=\sum\limits_{i=0}^{\infty}{u_i
L_{i,\sigma}(t)},\quad
u_i=\frac{(u,L_{i,\sigma})}{\|L_{i,\sigma}\|^2}.
\end{equation*}

Therefore, we have
\begin{equation*}\label{eqrev5}
\mathrm{M}_{N,\sigma}(\mathrm{V})=\text{Span}\big\{L_{0,\sigma}(t),
L_{1,\sigma}(t),...,L_{N,\sigma}(t)\big\}.\end{equation*}

The M\"untz-Jacobi orthogonal projection $\Pi_{N,\sigma}:L^2(\Omega) \to \mathrm{M}_{N,\sigma}(\mathrm{V})$, is defined by
$$\Pi_{N,\sigma} u(t)=\sum\limits_{i=0}^{N}{u_i
L_{i,\sigma}(t)},$$
which has the following property
\begin{equation}\label{eqrev6}
\big(\Pi_{N,\sigma} u-u,v \big)=0,~~\forall v \in
\mathrm{M}_{N,\sigma}(\mathrm{V}).
\end{equation}

In the sequel, we investigate the approximation properties of the M\"untz-Jacobi orthogonal projection $\Pi_{N,\sigma}$ by providing some estimations for the truncation error in $L_\infty$ and $L^2$ norms. Assume that the functions $u(t)$ and $U(s)$ are related by the coordinate transformation $s=t^\sigma$. Thus, their derivatives are connected as follows
\begin{eqnarray*}
D_t u:=\partial_s U(s)&=&\partial_t u~ \partial_s t,\\
D_t^2 u:=\partial_s^2 U(s)&=&\partial_t(D_t u)~\partial_s t,\\
\vdots
\\
D_t^n u:=\partial_s^n U(s)&=&\partial_t(\partial_s t~ \partial_t(\cdots(\partial_t u~ \partial_s t)\cdots))\partial_s t,
\end{eqnarray*}
and the following relations hold
\begin{eqnarray}\label{erev7}
\nonumber &&\|U(s)\|_{w^{0,\frac{1}{\sigma}-1}}^2=\int_{\Omega}{|U(s)|^2w^{0,\frac{1}{\sigma}-1}(s)ds}=
\sigma \int_{\Omega}{|u(t)|^2dt}=\sigma \|u(t)\|^2,\\
\nonumber \\
\nonumber &&\|\partial_s^k U(s)\|_{w^{0,\frac{1}{\sigma}-1}}^2=\int_{\Omega}{|\partial_s^k U(s)|^2w^{0,\frac{1}{\sigma}-1}(s)ds}=
\sigma \int_{\Omega}{|D_t^k u(t)|^2 dt}=\sigma \|D_t^k u(t)\|^2.
\end{eqnarray}

In the following lemmas we present the truncation error $\Pi_{N,\sigma} u - u$ in the $L^\infty$ and $L^2$-norms.
\begin{lemma}\label{lemrev3}
Let $D_t^k u\in L^2(\Omega),~k \geq 1$. Then there exists a positive constant $C$ independent of $N$ such that
\begin{equation}\label{erev8}
\|u-\Pi_{N,\sigma}u\|_{\infty} \leq C (1+N^{\frac{1}{\sigma}-\frac{1}{2}})N^{\frac{3}{4}-k}\|D_t^k u\|,
\end{equation}
\end{lemma}
\begin{proof}
Using the variable transformation $s=t^\sigma$, the following relation is obtained
\begin{equation}\label{erev10}
\|u-\Pi_{N,\sigma}u\|_\infty=\|U-\Pi_N^{0,\frac{1}{\sigma}-1}U\|_\infty,
\end{equation}

Since $\frac{1}{\sigma}-1>0$, the desired inequality is obtained using Lemma \ref{lemrev1} for $\vartheta, \xi >0$, and the relation \eqref{erev7}.
\end{proof}
\begin{lemma}\cite{mokhs2,Shenm} \label{lemrev4}
Assume that $D_t^k u \in L_{\tilde{w}^{k,k}}^{2}(\Omega)$ with $\tilde{w}^{k,k}(t)=w^{k,k}(t^\sigma)$, and $k \geq 1$. Then there exists a positive constant $C$ independent of $N$ such that
\begin{equation}\label{erev9}
\|u-\Pi_{N,\sigma}u\| \leq C  N^{-k}\|D_t^k u\|_{\tilde{w}^{k,k}}.
\end{equation}
\end{lemma}
\section{Existence, Uniqueness and Smoothness Results}\label{sec1}
It can be shown that the equation \eqref{1} is equivalent to the following
system of  equations
\begin{eqnarray}\label{1rev}
\nonumber Y(t)&=&G(t)+\int\limits_{0}^{t}{{K}(t,s)Y(s)ds}\\
&=&G(t)+\int\limits_{0}^{t}{(t-s)^{\alpha-1}\tilde{K}(t,s)Y(s)ds},
\end{eqnarray}
where $\alpha=\min\limits_{1 \le i,j \le n}{\{\alpha_{ij}\}}$, and
\begin{eqnarray*}
Y(t)&=&[y_{1}(t),y_{2}(t),...,y_{n}(t)]^{T},~~
G(t)=[g_{1}(t),g_{2}(t),...,g_{n}(t)]^{T},
\\
\\
{K}(t,s)&=&{\left[
  \begin{array}{ccc}
    (t-s)^{\alpha_{11}-1}k_{11}(t,s)&\ldots &(t-s)^{\alpha_{1n}-1}k_{1n}(t,s) \\
     (t-s)^{\alpha_{21}-1}k_{21}(t,s)&\ldots &(t-s)^{\alpha_{2n}-1}k_{2n}(t,s) \\
    \vdots & \vdots & \vdots \\
      (t-s)^{\alpha_{n1}-1}k_{n1}(t,s) &\ldots&   (t-s)^{\alpha_{nn}-1}k_{nn}(t,s) \\
  \end{array}
\right]}_{n\times n},\\
\\
\\
\tilde{K}(t,s)&=&(t-s)^{1-\alpha} K(t,s).
\end{eqnarray*}

Clearly, the vector-valued form of Theorem 4.8 in  \cite{linz}, concludes the following result regarding the existence and uniqueness of  solution for Eq. \eqref{1rev}.
\begin{theorem}[Existence and uniqueness of the solutions] \label{theo1}
Assume that the functions $k_{ij}(t,s)$ and $g_i(t)$, for $i,j=1,2,...,n$, are continuous on $D=\{(t,s): 0 \le s \le t \le 1\}$ and $\Omega$ respectively. Then the equation \eqref{1rev} has a unique continuous solution $Y(t)$ on $\Omega$.
\end{theorem}

To ensure the smoothness properties of \eqref{1}, we give the following theorem regarding the degree of regularity for the solution of the equation \eqref{1}
\begin{theorem}[Smoothness of the solutions]\label{th202} Under the assumptions of Theorem \ref{theo1}, assume that we can write
\begin{eqnarray}\label{eqrev3}
\nonumber && g_i(t)=\bar{g}_i(t^{1/\gamma}), \\
&& k_{ij}(t,s)=\bar{k}_{ij}(t^{1/\gamma},s^{1/\gamma}),\quad  i,j=1,2,...,n,
\end{eqnarray}
where $\gamma$ indicating the least common multiple of $\{b_{ij}\}_{i,j=1}^n$.  The functions $\bar{g}_i(t)$, and $\bar{k}_{ij}(t,s)$ are analytic in a neighborhood of zero and $(0,0)$, respectively. Then the series representation of the solution $y_i(t)$ of the equation \eqref{1} in a neighborhood of the origin is given by
\begin{equation}\label{eq3}
y_i(t)=\sum\limits_{\mu=0}^{\infty}{\bar{y}_{i,\mu}~t^{\frac{\mu}{\gamma}}},\quad i=1,2,...,n,
\end{equation}
where $\bar{y}_{i,\mu}$ are known coefficients.
\end{theorem}
\begin{proof}
First, we intend to obtain the unknown coefficients $\{\bar{y}_{i,\mu}\}_{i=1}^{n}$  such  that the representation \eqref{eq3} converges and solves \eqref{1}. Due to \eqref{eqrev3}, we have
\begin{eqnarray}\label{e54-11}
\nonumber &&g_i(t)=\sum\limits_{\mu=0}^{\infty}{\bar{g}_{i,\mu}~t^{\frac{\mu}{\gamma}}}, \\
&&k_{ij}(t,s)=\sum\limits_{\mu_{1}, \mu_2=0}^{\infty}{\bar{k}_{ij,\mu_{1},\mu_{2}}~t^{\frac{\mu_{1}}{\gamma}}s^{\frac{\mu_{2}}{\gamma}}},\quad i,j=1,2,...,n,
\end{eqnarray}

Meanwhile, by considering uniform convergence, inserting the relations \eqref{eq3} and \eqref{e54-11} into \eqref{1} and some simple manipulations, the coefficients $\{\bar{y}_{i,\mu}\}_{i=1}^{n}$ satisfy the following equality
\begin{equation}\label{eq148}
\sum\limits_{\mu=0}^{\infty}{\bar{y}_{i,\mu}~t^{\frac{\mu}{\gamma}}}=\sum\limits_{\mu=0}^{\infty}{\bar{g}_{i,\mu}~t^{\frac{\mu}{\gamma}}}+\sum\limits_{j=1}^{n}
~\sum\limits_{\mu_{1},\mu_{2}=0}^{\infty}~\sum\limits_{\mu=0}^{\infty}{\bar{k}_{ij,\mu_{1},\mu_{2}}~ \bar{y}_{j,\mu}~~\varrho_{ij}~~t^{\frac{\mu_{1}+\mu_{2}+\mu}
{\gamma}+\alpha_{ij}}},
\end{equation}
in which $\varrho_{ij}=\beta(\alpha_{ij},\frac{\mu_{2}+\mu}{\gamma}+1)$, and $\beta(..)$ denotes the well known Beta function. Substituting $\mu=\mu-\mu_{1} -\mu_{2}-\alpha_{ij} \gamma$ in the second series of the right-hand side of \eqref{eq148} yields
\begin{equation}\label{eq149}
\sum\limits_{\mu=0}^{\infty}{\bar{y}_{i,\mu}~t^{\frac{\mu}{\gamma}}}=\sum\limits_{\mu=0}^{\infty}{\bar{g}_{i,\mu}~t^{\frac{\mu}{\gamma}}}+\sum\limits_{j=1}^{n}~
\sum\limits_{\mu_{1},\mu_{2}=0}^{\infty}~\sum\limits_{\mu=\mu_{1}+\mu_{2}+\alpha_{ij} \gamma}^{\infty}{\bar{k}_{ij,\mu_{1},\mu_{2}}~~\bar{y}_{j,\mu-\mu_{1} -\mu_{2}-\alpha_{ij} \gamma}~~ \bar{\varrho}_{ij} ~t^{\frac{\mu}{\gamma}}},
\end{equation}
where $\bar{\varrho}_{ij}=\beta(\alpha_{ij},\frac{\mu-\mu_{1}}{\gamma}-\alpha_{ij}+1)$. In this step, we compare the coefficients of $t^{\frac{\mu}{\gamma}}$ on both sides of \eqref{eq149} and evaluate the unknown coefficients $\bar{y}_{i,\mu}$. Evidently, for $\mu < \min\limits_{1 \leq l \leq n}\{\alpha_{il}\gamma\}$, we have $\bar{y}_{i,\mu}=\bar{g}_{i,\mu}$, and for $\mu \geq \min\limits_{1 \leq l \leq n}\{\alpha_{il}\gamma\}$, the following recursive relation is derived
\[
\bar{y}_{i,\mu}=\bar{g}_{i,\mu}+\sum\limits_{j=1}^{n}~\sum\limits_{\mu_{1},\mu_{2}=0}^{\infty}{\bar{k}_{ij,\mu_{1},\mu_{2}}~~\bar{y}_{j,\mu-\mu_{1} -\mu_{2}-\alpha_{ij} \gamma}~~ \bar{\varrho}_{ij} },
\]
such that the coefficients with negative indices are considered as zero. This deduces that the series representation \eqref{eq3} solves the main problem \eqref{1} uniquely.

Now, we should prove that the series \eqref{eq3} converges uniformly and absolutely in a neighborhood of origin. To this end, a generalization of Lindelof's majorant method \cite{43,39} is utilized. Let us consider the following system of generalized Abel-Volterra integral equations
\[
Y_i(t)=G_i(t)+\sum\limits_{j=1}^{n}{\int_{0}^{t}(t-s)^{\alpha_{ij}-1}K_{ij}(t,s)Y_{j}(s)ds},~~i=1,2,\ldots,n,
\]
where $G_i(t)$, and $K_{ij}(t,s)$ are majorant of $g_i(t)$ and $k_{ij}(t,s)$, respectively which are obtained by taking absolute values of the coefficients in \eqref{e54-11}. Clearly, $\{Y_i(t)\}_{i=1}^n$ is a majorant for $\{y_i(t)\}_{i=1}^n$, and all of it's coefficients $\bar{Y}_{i,\mu}$ are positive. The formal solution $\{Y_i(t)\}_{i=1}^n$ can be evaluated in exactly the same way as the previous step. Now, we show that the series $Y_i(t)$ is absolutely convergent on $[0,\varepsilon_i]$, with some $\varepsilon_i>0$ which is given in the sequel. To do this, it suffices to prove that the finite partial sum of $Y_i(t)$ is uniformly bounded over $[0,\varepsilon_i]$. Let
\[
S_{i,L+1}(t)=\sum\limits_{\mu=0}^{L+1}{\bar{Y}_{i,\mu}~t^{\frac{\mu}{\gamma}}},
\]
be the finite partial sum of $Y_i(t)$ for $i =1,2,\ldots,n$. Evidently, the following inequality holds
\begin{equation*}
S_{i,L+1}(t) \le G_i(t)+\sum\limits_{j=1}^{n}{\int_{0}^{t}(t-s)^{\alpha_{ij}-1}K_{ij}(t,s)S_{j,L}(s)ds},~~i=1,2,\ldots,n,
\end{equation*}
due to the recursive calculation of the coefficients. Indeed, if we expand the right-hand side of the above inequality, all coefficients $\bar{Y}_{i,\mu}$ with $\frac{\mu}{\gamma}\le \frac{(L+1)}{\gamma}$ are eliminated from both sides, while there will be some additional positive terms of higher-order in the right-hand side.

Considering
\begin{eqnarray*}
D_{i}^{1}&=&\max\limits_{t\in \Omega}(G_{i}(t)),\\
D_{i}^{2}&=&\frac{2}{\alpha} \sum_{j=1}^{n}{D_j^1 \max\limits_{(t,s) \in \Omega \times \Omega}{(K_{ij}(t,s))}},~~i=1,2,\ldots,n,
\end{eqnarray*}
we define
$$\varepsilon_i=\min\bigg\{1,\bigg[\frac{D_{i}^{1}}{D_{i}^{2}}\bigg]^{\frac{1}{\alpha}}\bigg\},~~i=1,2,\ldots,n.$$

Now, we intend to prove that $$|S_{i,L}(t)| \le 2 D_{i}^{1},~~t \in [0,\varepsilon_i]~~i=1,2,\ldots,n.$$

To this end, we apply the principle of mathematical induction on $L$. For $L=0$, it is evident since
\[
S_{i,0}(t)=|\bar{g}_{i,0}| \leq D_{i}^{1},~~i=1,2,\ldots,n.
\]

We will show that if the statement holds up for $L$, then it  also holds up for $L+1$. We have
\begin{eqnarray*}
|S_{i,L+1}(t)|&=&S_{i,L+1}(t) \le G_i(t)+\sum\limits_{j=1}^{n}{\int_{0}^{t}(t-s)^{\alpha_{ij}-1}K_{ij}(t,s)S_{j,L}(s)ds} \\
&=& G_i(t)+\int_{0}^{t}(t-s)^{{\alpha}-1}\bigg(\sum\limits_{j=1}^{n}{(t-s)^{\alpha_{ij}-{\alpha}}K_{ij}(t,s)S_{j,L}(s)\bigg)ds}\\
  &\le & \max\limits_{t\in [0,\varepsilon_{i}]}(G_{i}(t))+\max_{s \in [0,t]}{\bigg(\sum_{j=1}^{n}(t-s)^{\alpha_{ij}-{\alpha}}K_{ij}(t,s)S_{j,L}(s)\bigg)} \frac{t^{{\alpha}}}{{\alpha}}\\
  &\le &D_{i}^{1}+2 \max_{(t,s)\in [0,\varepsilon_{i}] \times [0,\varepsilon_{i}] }{\bigg(\sum_{j=1}^{N}(t-s)^{\alpha_{ij}-{\alpha}}K_{ij}(t,s)D_{j}^1\bigg)} \frac{\varepsilon_{i}^{{\alpha}}}{{\alpha}}\\
  &\le & D_{i}^{1}+\varepsilon_{i}^{{\alpha}} D_{i}^{2}\le 2 D_{i}^{1},~~i =1,2,\ldots,n,
\end{eqnarray*}
which deduces the uniform boundedness of $S_{i,L+1}(t)$ over interval $[0, \varepsilon_{i}]$. In view of the positivity of all it's coefficients, it is also monotone. Hence, since $Y_i(t)$ has a power series structure, it is absolutely and uniformly convergent on $[0,\varepsilon_{i}]$, and the compact subsets of $[0,\varepsilon_{i})$, respectively. Finally, using Lindelof's theorem, it ultimately yields the same properties for series representation of m$y_i(t)$. Thereby, the interchange of integration and series was done properly.
\end{proof}

From Theorem \ref{th202}, we can conclude that the first derivative of $y_{i}(x)$ often has a discontinuity at the origin. This difficulty affects accuracy when the classical  spectral methods are used to approximate the solutions of \eqref{1}. In this regard, we intend to introduce a new recursive Tau method that produces reliable approximations for the solutions of \eqref{1} regardless of their non-smooth behaviors. The remainder of the paper is devoted to establish a new high accurate numerical approach to approximate the solutions of \eqref{1}.
\section{Numerical Approach}\label{sec2}
\subsection{Review of the recursive Tau method}
In this section, we present a brief review for the implementation process of the recursive Tau approach for the numerical solution of the following integer order differential equation
\begin{equation}\label{ODE}
\left\{
  \begin{array}{ll}
\mathcal{L}u(t)=\displaystyle\sum_{i=0}^{l}p_{i}(t)u^{(i)}(t)=f(t),\ \ \ \ t\in \Omega,\\
B_{j}(u)=d_{j},\ \ \ \  j=1,...,l,
  \end{array}
\right.
\end{equation}
where $B_{j},\ j=1,...,l$ are linear functionals and $p_{i}(t)$ and $f(t)$ are polynomials.

 Lanczos proposed the idea of Tau method in \cite{Lanczos2}, to find an exact  polynomial solution $u_{N}(t)$ by adding a perturbation term $H_{N}(t):=\sum_{i=N-l+1}^{N+\hbar}\tau_{i}v_{i}(t)$ to (\ref{ODE}), i.e.,
 \begin{equation*}\label{ODEh}
 \left\{
  \begin{array}{ll}
\mathcal{L}u_{N}(t)=f(t)+H_{N}(t),\ \ \ \ t\in \Omega,\\
B_{j}(u_{N})=d_{j},\ \ \ \  j=1,...,l,
  \end{array}
\right.
\end{equation*}
where $N$ is the degree of approximation, $\hbar:=\max \lbrace deg(\mathcal{L} t^i)-i;\ i\geq 0\rbrace $ and $v_{i}(t)=\sum\limits_{j=0}^{i}c^{(i)}_{j}t^{j}$ are orthogonal basis polynomials. The parameters $\tau_{i},\ i=0,...,N+\hbar$ are unknown parameters. To explain the recursive Tau method, Ortiz \cite{Ortiz1} defined the set of canonical polynomials $Q_{i}(t)$ as follows
\begin{equation*}\label{xanon}
\mathcal{L}Q_{i}(t)=t^{i}+R_{i}(t),\ \ \  i \in \mathbb{N}_{0}-S,
\end{equation*}
 where $S$ is the set of all $i$ for which $t^{i}$ is inaccessible through applying  $\mathcal{L}$ to a polynomial and $R_{i}(t) \in \text{Span} \lbrace t^{i}: i\in S \rbrace$ are the residual polynomials of the canonical polynomials $Q_{i}(t)$. Let $ f(t)=\sum\limits_{i=0}^{F}f_{i}t^{i}$ and the set $ \lbrace E_{1}(t), E_{2}(t),..., E_{\varpi}(t)\rbrace$  is a polynomial kernel basis for the operator $\mathcal{L}$. Then,
\begin{equation*}\label{tttt}
u_{N}(t)=\sum\limits_{\underset{i\notin S}{i=0}}^{F}f_{i}Q_{i}(t)+\sum\limits_{i=N-l+1}^{N+\hbar}\tau_{i}\left(\sum\limits_{\underset{j\notin S}{j=0}}^{i}c^{(i)}_{j}Q_{j}(t)\right)+\sum\limits_{i=1}^{\varpi}\tau_{N+\hbar+i}d_{i}(t)
\end{equation*}
is called recursive Tau-approximate solution for (\ref{ODE}). The  $\hbar+l+\varpi $ unknown
$\tau$-parameters are chosen by $\hbar+\varpi$ conditions  imposed by  the recursive Tau method and $l$ supplementary conditions of (\ref{ODE}).
\subsection{New fractional vector canonical polynomials}
As we proved in the previous section, the first derivative of the solutions of \eqref{1} typically has a discontinuity at origin. This weakness makes it impossible to implement the classical approach of the recursive Tau method to obtain a suitable accurate approximation. To fix this problem, in this section, we introduce and develop a new recursive Tau approach by producing new fractional vector canonical polynomials and providing a high order approximate solution. In this regards, we set $\delta_{ij}=\gamma\alpha_{ij}$, $i,j=1,2,...,n$ and $\sigma:=\frac{1}{\gamma}$, and define
\begin{equation}\label{014}
LY(t):=Y(t)-\int_{0}^{t}K(t,s)Y(s)ds.
\end{equation}
Clearly, from \eqref{1rev} we have
\begin{equation}\label{014rev}
 L Y(t)=G(t).
\end{equation}
\begin{definition}\label{D1}
A vector polynomial ${\bf{Q}}_{i}^{j}(t)$ is called fractional vector canonical polynomial of $L$, if
\[
L({\bf{Q}}_{i}^{j}(t))=t^{j\sigma}{\bf{e}}_{i}+{\bf{R}}_{i}^{j}(t), \ \ \ j\in {\mathbb{N}}_{0}-S_{i}, \ \ i=1,...,n,
\]
where ${\bf{R}}_{i}^{j}(t)$ is called vector residual polynomial of ${\bf{Q}}_{i}^{j}(t)$. When $j \in S_{i}$, we say that ${\bf{Q}}_{i}^{j}(t)$  does  not exist.
Note that,  ${\bf{e}}_{i}=[0,0,\cdots,1,\cdots,0]^{T}$ are unit vectors in $\mathbb{R}^{n}$ for $i=1,...,n$.
\end{definition}
\begin{definition}\label{D2}
The vector space of polynomials
\[
\mathcal{R}(t)=\text{Span} \left \{ \left[
  \begin{array}{c}
    t^{j\sigma} \\
    0\\
    \vdots\\
    0 \\
  \end{array}
\right]
,\left[
  \begin{array}{c}
  0\\
    t^{j\sigma} \\
    \vdots\\
    0 \\
  \end{array}
\right],...,\left[
  \begin{array}{c}
  0\\
\vdots\\
    0\\
       t^{j\sigma} \ \\
  \end{array}
\right];\ j\in \bigcup\limits_{i=1}^n S_{i} \right \}.
\]
 is called the residual space of $L$.
\end{definition}

 \begin{theorem}\label{Int}
Assume that $\underline{\bf{X}}^{\sigma}_{t}=[1,t^{\sigma},t^{2\sigma},...]^{T}$ and $k_{ij}(t,s)\simeq \sum\limits_{p,q=0}^{N}\widehat{k}^{ij}_{p,q}t^{p
\sigma}s^{q \sigma}$. Then,
\[
\int\limits_{0}^{t}(t-s)^{\alpha_{ij}-1}k_{ij}(t,s)s^{r\sigma}ds \simeq {\bf{e}}^{T}_{r+1}\ \Lambda_{ij} \
\underline{\bf{X}}^{\sigma}_{t},\ \ \ \ r=0,1,2,..., \ \ i,j=1,...,n,
\]
where $ \Lambda_{ij}$ is an infinite matrix defined as
 \begin{equation}\label{fd3}
\Lambda_{ij}=\left[
  \begin{array}{ccccccccccc}
    \overbrace{0 \cdots 0}^{\delta_{ij}} &  \widetilde{k}_{0,0}^{ij} & \widetilde{k}_{1,0}^{ij} & \cdots &  \widetilde{k}_{2N,0}^{ij}&0&\cdots &\cdots & \cdots\\
     \vdots &  0 & \widetilde{k}_{0,1}^{ij} & \widetilde{k}_{1,1}^{ij}&\cdots &  \widetilde{k}_{2N,1}^{ij}&0& \cdots &\cdots\\
      \vdots &\vdots &0& \widetilde{k}_{0,2}^{ij} & \widetilde{k}_{1,2}^{ij}&\cdots &  \widetilde{k}_{2N,2}^{ij}&0 &\cdots \\
        \vdots &\vdots &\vdots & \vdots & \vdots & \vdots & \ddots &\ddots &\ddots & \ddots &\vdots\\
  \end{array}
\right]
\end{equation}
and
\begin{equation}\label{fd2}
 \widetilde{k}_{v,l}^{ij}=\sum_{\underset{(p+q=v)}{0 \leq p,q \leq N}}\overline{k}^{ij}_{p,q,l}; \ \ \  \ \ \ \overline{k}^{ij}_{p,q,l}=\widehat{k}^{ij}_{p,q}\beta\left(\delta_{ij}\sigma, (q+l)\sigma +1\right),~~v, l \ge 0.
\end{equation}
\end{theorem}
\begin{proof}
From $s^{r\sigma}={\bf{e}}^{T}_{r+1}\underline{\bf{X}}^{\sigma}_{s}$, we have
\begin{eqnarray}
\int\limits_{0}^{t}(t-s)^{\alpha_{ij}-1}k_{ij}(t,s)s^{r\sigma}ds &\simeq & {\bf{e}}^{T}_{r+1}\sum_{p,q=0}^{N}\widehat{k}^{ij}_{p,q}t^{p\sigma}\int\limits_{0}^{t}(t-s)^{\delta_{ij}\sigma -1}s^
{q\sigma }\  \underline{\bf{X}}^{\sigma}_{s} ds\nonumber\\
&=&{\bf{e}}^{T}_{r+1}\left[
  \begin{array}{c}
   \displaystyle\sum_{p,q=0}^{N}\widehat{k}^{ij}_{p,q}\beta\left(\delta_{ij}\sigma, q\sigma +1\right)t^{(p+q+\delta_{ij})\sigma}\\
    \displaystyle\sum_{p,q=0}^{N}\widehat{k}^{ij}_{p,q}\beta\left(\delta_{ij}\sigma, (q+1)\sigma +1\right)t^{(p+q+\delta_{ij}+1)\sigma} \\
    \vdots \\
  \end{array}
\right]\nonumber\\
&=&{\bf{e}}^{T}_{r+1}\left[
  \begin{array}{c}
  \displaystyle\sum_{p,q=0}^{N}\overline{k}^{ij}_{p,q,0}\ t^{(p+q+\delta_{ij})\sigma}\\
   \displaystyle \sum_{p,q=0}^{N}\overline{k}^{ij}_{p,q,1}\ t^{(p+q+\delta_{ij}+1)\sigma} \\
    \vdots \\
  \end{array}
\right]\nonumber\\
&=&{\bf{e}}^{T}_{r+1}\ \Lambda_{ij} \ \underline{\bf{X}}^{\sigma}_{t},\nonumber
\end{eqnarray}
which completes the proof.
\end{proof}
\begin{definition}
The height vector of $L$ is defined by
\begin{equation}\label{12345}
\mathbf{h}=[h_{1},h_{2},...,h_{n}],\ \ \
h_{i}=\underset{1\leq j\leq n}{max}\lbrace h_{ij} \rbrace, \ \ \ i=1,...,n,\
\end{equation}
where $h_{ij}$  is the number of non-zero diagonals above the main diagonal of $\Lambda_{ij}$.
\end{definition}
Now, we provide a recursive algorithm in order to generate  the vector polynomials ${\bf{Q}}_{j}^{i}(t)$ and their associated residuals ${\bf{R}}_{j}^{i}(t)$:
\begin{theorem}\label{newcan}
Let $\triangle_{j}:=\underset{1\leq i \leq n}{min}\lbrace h_{i}-h_{ij} \rbrace,\ j=1,...,n$. Then,
\begin{equation}\label{12341}
\left\{
  \begin{array}{ll}
&{\bf{Q}}_{j}^{h_{j}+r}(t)=\displaystyle\sum_{i=1}^{n}d_{ij}\left(t^{(r+\triangle_{i})\sigma}{\bf{e}}_{i}-
\sum_{v=1}^{n} \displaystyle\sum_{l=1}^{r+h_{v}}\widetilde{\Lambda}_{v,i}(r+\triangle_{i}+1,l){\bf{Q}}_{v}^{l-1}(t)
\right),\\
&{\bf{R}}_{j}^{h_{j}+r}(t)=-\displaystyle\sum_{i=1}^{n}d_{ij}\left(\sum_{v=1}^{n} \displaystyle\sum_{l=1}^{r+h_{v}}\widetilde{\Lambda}_{v,i}(r+\triangle_{i}+1,l)
{\bf{R}}_{v}^{l-1}(t)\right),
\end{array}
\right.
\end{equation}
 for $j=1,...,n$, $r=0,1,2,...$,  where $d_{ij}$ denote the elements of matrix $D:=P_{r}^{-1}$ in which
\[P_{r}:=\left[
  \begin{array}{ccc}
    \widetilde{\Lambda}_{1,1}(r+\triangle_{1}+1,r+h_{1}+1) &\cdots & \widetilde{\Lambda}_{1,n}(r+\triangle_{n}+1,r+h_{1}+1) \\
      \widetilde{\Lambda}_{2,1}(r+\triangle_{1}+1,r+h_{2}+1) & \cdots & \widetilde{\Lambda}_{2,n}(r+\triangle_{n}+1,r+h_{2}+1) \\
    \vdots & \vdots & \vdots \\
      \widetilde{\Lambda}_{n,1}(r+\triangle_{1}+1,r+h_{n}+1) & \cdots & \widetilde{\Lambda}_{n,n}(r+\triangle_{n}+1,r+h_{n}+1)
  \end{array}
\right],
\]
with
\[ \widetilde{\Lambda}_{i\ell}:=\left\{
  \begin{array}{ll}
{\bf{I}}-\ \Lambda_{{\ell\ell}}, &\ \  i= \ell, \\
- \Lambda_{{i\ell}}, &\ \  i\neq \ell,
  \end{array}
\right.
\]
for $i,\ell=1,...,n$ and $ \widetilde{\Lambda}_{i,j}(r,l)$ refers to $(r,l)$-entry of the matrix $\widetilde{\Lambda}_{i,j}$.
\end{theorem}
\begin{proof}From Theorem \ref{Int}, we can get
\begin{eqnarray}\label{1234}
L(t^{r\sigma}{\bf{e}}_{\ell})&=&t^{r\sigma}{\bf{e}}_{\ell}-\int_{0}^{t}K(t,s)s^{r\sigma}{\bf{e}}_{\ell}ds\nonumber\\
&=& \left[
  \begin{array}{c}
    -\displaystyle\int_{0}^{t}(t-s)^{\alpha_{1\ell}-1}k_{1\ell}(t,s)s^{r\sigma}ds   \\
   \vdots \\
 t^{r\sigma}-\displaystyle\int_{0}^{t}(t-s)^{\alpha_{\ell\ell}-1}k_{\ell\ell}(t,s)s^{r\sigma}ds  \\
   \vdots \\
     -\displaystyle\int_{0}^{t}(t-s)^{\alpha_{n\ell}-1}k_{n\ell}(t,s)s^{r\sigma}ds  \\
  \end{array}
\right]=\left[
  \begin{array}{c}
    -{\bf{e}}^{T}_{r+1}\ \Lambda_{{1\ell}}\underline{\bf{X}}^{\sigma}_{t}  \\
   \vdots \\
    {\bf{e}}^{T}_{r+1}({\bf{I}}-\ \Lambda_{{\ell\ell}}) \
\underline{\bf{X}}^{\sigma}_{t} \\
   \vdots \\
     -{\bf{e}}^{T}_{r+1}\ \Lambda_{{n\ell}}\underline{\bf{X}}^{\sigma}_{t}  \\
  \end{array}
\right]=\left[
  \begin{array}{c}
    {\bf{e}}^{T}_{r+1}\ \widetilde{\Lambda}_{{1\ell}}\underline{\bf{X}}^{\sigma}_{t}  \\
   \vdots \\
    {\bf{e}}^{T}_{r+1}\widetilde{\Lambda}_{{\ell\ell}} \
\underline{\bf{X}}^{\sigma}_{t} \\
   \vdots \\
     {\bf{e}}^{T}_{r+1}\ \widetilde{\Lambda}_{{n\ell}}\underline{\bf{X}}^{\sigma}_{t}  \\
  \end{array}
\right]\nonumber\\
\end{eqnarray}
 for $\ell=1,...,n$ and $r\geq 0$. Then,  by replacing  $t^{(r+\triangle_{j})\sigma}\rightarrow t^{r\sigma}$ in (\ref{1234}) we obtain
\begin{align}\label{rece}
&\left(L(t^{(r+\triangle_{1})\sigma}{\bf{e}}_{1}),...,L(t^{(r+\triangle_{n})\sigma}{\bf{e}}_{n}) \right)=\left[
  \begin{array}{cccc}
  {\bf{e}}^{T}_{r+\triangle_{1}+1}\ \widetilde{\Lambda}_{{11}}\underline{\bf{X}}^{\sigma}_{t}   & \cdots & \cdots & {\bf{e}}^{T}_{r+\triangle_{n}+1}\ \widetilde{\Lambda}_{{1n}}\underline{\bf{X}}^{\sigma}_{t}\\\\
    {\bf{e}}^{T}_{r+\triangle_{1}+1}\ \widetilde{\Lambda}_{{21}}\underline{\bf{X}}^{\sigma}_{t}   & \cdots& \cdots & {\bf{e}}^{T}_{r+\triangle_{n}+1}\ \widetilde{\Lambda}_{{2n}}\underline{\bf{X}}^{\sigma}_{t}\\
    \vdots & \vdots & \vdots & \vdots \\
     {\bf{e}}^{T}_{r+\triangle_{1}+1}\ \widetilde{\Lambda}_{{n1}}\underline{\bf{X}}^{\sigma}_{t}   & \cdots & \cdots & {\bf{e}}^{T}_{r+\triangle_{n}+1}\ \widetilde{\Lambda}_{{nn}}\underline{\bf{X}}^{\sigma}_{t}\\
  \end{array}
\right]\nonumber \\
&=\displaystyle \left[
  \begin{array}{ccc}
   \displaystyle\sum_{l=1}^{r+h_{1}+1}\widetilde{\Lambda}_{1,1}(r+\triangle_{1}+1,l)t^{(l-1)\sigma} & \cdots & \displaystyle\sum_{l=1}^{r+h_{1}+1}\widetilde{\Lambda}_{1,n}(r+\triangle_{n}+1,l)t^{(l-1)\sigma} \\
      \displaystyle\sum_{l=1}^{r+h_{2}+1}\widetilde{\Lambda}_{2,1}(r+\triangle_{1}+1,l)t^{(l-1)\sigma} & \cdots & \displaystyle\sum_{l=1}^{r+h_{2}+1}\widetilde{\Lambda}_{2,n}(r+\triangle_{n}+1,l)t^{(l-1)\sigma} \\
  \vdots   & \vdots & \vdots \\
      \displaystyle\sum_{l=1}^{r+h_{n}+1}\widetilde{\Lambda}_{n,1}(r+\triangle_{1}+1,l)t^{(l-1)\sigma} & \cdots & \displaystyle\sum_{l=1}^{r+h_{n}+1}\widetilde{\Lambda}_{n,n}(r+\triangle_{n}+1,l)t^{(l-1)\sigma}
  \end{array}
\right]\nonumber \\
\nonumber\\
&=\left[
  \begin{array}{ccc}
    t^{(r+h_{1})\sigma} & \ldots & 0 \\
    \vdots & \ddots & \vdots \\
    0 & \cdots  & t^{({r+h_{n}})\sigma}  \\
  \end{array}
\right].
 \left[
  \begin{array}{ccc}
    \widetilde{\Lambda}_{1,1}(r+\triangle_{1}+1,r+h_{1}+1) &\cdots & \widetilde{\Lambda}_{1,n}(r+\triangle_{n}+1,r+h_{1}+1) \\
      \widetilde{\Lambda}_{2,1}(r+\triangle_{1}+1,r+h_{2}+1) & \cdots & \widetilde{\Lambda}_{2,n}(r+\triangle_{n}+1,r+h_{2}+1) \\
    \vdots & \vdots & \vdots \\
      \widetilde{\Lambda}_{n,1}(r+\triangle_{1}+1,r+h_{n}+1) & \cdots & \widetilde{\Lambda}_{n,n}(r+\triangle_{n}+1,r+h_{n}+1) \\
  \end{array}
\right]
\nonumber \\  \nonumber\\
&+\displaystyle \left[
  \begin{array}{cccc}
   \displaystyle\sum_{l=1}^{r+h_{1}}\widetilde{\Lambda}_{1,1}(r+\triangle_{1}+1,l)t^{(l-1)\sigma} & \cdots & \displaystyle\sum_{l=1}^{r+h_{1}}\widetilde{\Lambda}_{1,n}(r+\triangle_{n}+1,l)t^{(l-1)\sigma} \\
      \displaystyle\sum_{l=1}^{r+h_{2}}\widetilde{\Lambda}_{2,1}(r+\triangle_{1}+1,l)t^{(l-1)\sigma} & \cdots & \displaystyle\sum_{l=1}^{r+h_{2}}\widetilde{\Lambda}_{2,n}(r+\triangle_{n}+1,l)t^{(l-1)\sigma} \\
    \vdots & \vdots & \vdots \\
      \displaystyle\sum_{l=1}^{r+h_{n}}\widetilde{\Lambda}_{n,1}(r+\triangle_{1}+1,l)t^{(l-1)\sigma} & \cdots & \displaystyle\sum_{l=1}^{r+h_{n}}\widetilde{\Lambda}_{n,n}(r+\triangle_{n}+1,l)t^{(l-1)\sigma}
  \end{array}
\right].
\end{align}
Now, multiply both sides of (\ref{rece}) from right by $D=[d_{ij}]_{i,j=1}^{n}$. We obtain
\begin{equation}
\sum_{i=1}^{n}d_{ij}L(t^{(r+\triangle_{i})\sigma}{\bf{e}}_{i})=t^{(r+h_{j})\sigma}{\bf{e}}_{j}+\sum_{i=1}^{n}d_{ij}\left(\sum_{v=1}^{n} \displaystyle\sum_{l=1}^{r+h_{v}}\widetilde{\Lambda}_{v,i}(r+\triangle_{i}+1,l)t^{(l-1)\sigma} {\bf{e}}_{v}\right),\ \ \ 1 \leq j \leq n.
\end{equation}
Set
\begin{equation}\label{121212}
{\bf{Q}}_{i}^{j}(t)={\left[
  \begin{array}{c}
   0 \\
    0 \\
   \vdots \\
    \vdots \\
    0 \\
  \end{array}
\right],}_{n \times 1}
\ \ \ \ \ \ \ \ {\bf{R}}_{i}^{j}(t)={\left[
  \begin{array}{c}
  0\\
    0\\
    -t^{j\sigma}\\
    \vdots\\
    0 \\
  \end{array}
\right]_{n \times 1}\ \leftarrow   \textit{i-th component}}
\end{equation}
for $j\in S_{i}:=\lbrace 0,...,h_{i}-1\rbrace$ and $ i=1,...,n$, therefore
\begin{equation}\label{Jadidb}
L({\bf{Q}}_{i}^{j}(t))=t^{j\sigma}{\bf{e}}_{i}+{\bf{R}}_{i}^{j}(t), \ \ \ j \in S_{i}, \ \ i=1,...,n.
\end{equation}
From Definition \ref{D1} and relation (\ref{Jadidb}) one can write
\begin{align}
\sum_{i=1}^{n}d_{ij}L(t^{(r+\triangle_{i})\sigma}{\bf{e}}_{i})&=t^{(r+h_{j})\sigma}{\bf{e}}_{j}\nonumber\\
&+\sum_{i=1}^{n}d_{ij}\left(\sum_{v=1}^{n} \displaystyle\sum_{l=1}^{r+h_{v}}\widetilde{\Lambda}_{v,i}(r+\triangle_{i}+1,l)\left( L({\bf{Q}}_{v}^{l-1}(t))-{\bf{R}}_{v}^{l-1}(t)\right)\right),\nonumber
\end{align}
then a simple manipulation yields
\begin{align}
&L\left(\sum_{i=1}^{n}d_{ij}\left(t^{(r+\triangle_{i})\sigma}{\bf{e}}_{i}-
\sum_{v=1}^{n} \displaystyle\sum_{l=1}^{r+h_{v}}\widetilde{\Lambda}_{v,i}(r+\triangle_{i}+1,l){\bf{Q}}_{v}^{l-1}(t)
\right)\right)=t^{(r+h_{j})\sigma}{\bf{e}}_{j}\nonumber\\
&-\sum_{i=1}^{n}d_{ij}\left(\sum_{v=1}^{n} \displaystyle\sum_{l=1}^{r+h_{v}}\widetilde{\Lambda}_{v,i}(r+\triangle_{i}+1,l)
{\bf{R}}_{v}^{l-1}(t)\right),\ \ \ 1 \leq j \leq n.
\end{align}
and so Definition \ref{D1} gives (\ref{12341}), which completes the proof.
\end{proof}
\subsection{Construction of recursive  Tau-approximate solution}
Let us associate to Eq. (\ref{014rev}) a $\tau$-problem of the form
\begin{equation}\label{reee2}
LY_{N}(t)=G(t)+H_{N}(t),
\end{equation}
with the exact vector polynomial solution
\[Y_{N}(t)=[y_{N,1}(t),\ldots,y_{N,n}(t)]^{T} \simeq Y(t),
\]
and a perturbation term of the form
\begin{align}\label{hhhh}
H_{N}(t)&=[H_{N,1}(t),H_{N,2}(t),...,H_{N,n}(t)]^{T}\nonumber\\
&=\left[\sum_{j=N+1}^{N+h_{1}}\tau^{N}_{j,1}p_{j}(t),...,\sum_{j=N+1}^{N+h_{N}}\tau^{N}_{j,n}p_{j}(t)\right]^{T}\nonumber\\
&=\sum\limits_{i=1}^{n}\left(\sum\limits_{j=N+1}^{N+h_{i}}\tau^{N}_{j,i}\sum\limits_{\ell=0}^{j}c_{j,\ell}t^{\ell\sigma}\right){\bf{e}}_{i},
\end{align}
where $p_{j}(t)=\sum\limits_{\ell=0}^{j}c_{j,\ell}t^{\ell\sigma}$ are the orthonormal M\"untz-Legendre polynomials. Assume that
\begin{align}\label{fff1}
G(t) \simeq \left[\sum_{\ell=0}^{N}g_{\ell,1} t^{\ell\sigma},...,\sum_{\ell=0}^{N}g_{\ell,n} t^{\ell\sigma}\right]^{T}=\sum_{i=1}^{n}\sum_{\ell=0}^{N}g_{\ell,i}t^{\ell\sigma}{\bf{e}}_{i}.
\end{align}
Now, by substituting (\ref{hhhh}) and (\ref{fff1}) in (\ref{reee2})
and using Definition \ref{D1} and relation (\ref{Jadidb}) one can rewrite
\begin{align}\label{reee1}
L&\left(Y_{N}(t)-\sum_{i=1}^{n}\sum_{\ell=0}^{N}g_{\ell,i}{\bf{Q}}_{i}^{\ell}(t)-
\sum\limits_{i=1}^{n}\sum\limits_{j=N+1}^{N+h_{i}}\tau^{N}_{j,i}\sum\limits_{\ell=0}^{j}c_{j,\ell}{\bf{Q}}_{i}^{\ell}(t)\right)\nonumber\\=
&\underbrace{-\sum_{i=1}^{n}\sum_{\ell=0}^{N}g_{\ell,i}{\bf{R}}_{i}^{\ell}(t)-
\sum\limits_{i=1}^{n}\sum\limits_{j=N+1}^{N+h_{i}}\tau^{N}_{j,i}\sum\limits_{\ell=0}^{j}c_{j,\ell}{\bf{R}}_{i}^{\ell}(t)}_{:=R(t)}
\end{align}
Since $R(t)$ is in the residual space of  the operator $L$, then the relation (\ref{reee1}) is valid  if $R(t)\equiv {\bf{0}}$. Since $ker (L)=\lbrace \mathbf{0} \rbrace$, then we conclude
\begin{equation}\label{wet}
Y_{N}(t)=\sum_{i=1}^{n}\sum_{\ell=0}^{N}g_{\ell,i}{\bf{Q}}_{i}^{\ell}(t)+
\sum\limits_{i=1}^{n}\sum\limits_{j=N+1}^{N+h_{i}}\tau^{N}_{j,i}\sum\limits_{\ell=0}^{j}c_{j,\ell}{\bf{Q}}_{i}^{\ell}(t),
\end{equation}
that is called the recursive Tau-approximate solution of \eqref{1} with unknown $\tau$-parameters $\tau^{N}_{j,i}$ that are the solution of the obtained system of algebraic equations
\begin{equation}\label{javab}
R(t)\equiv {\bf{0}}.
\end{equation}

Note that, if the kernel functions $k_{ij}(t,s)$ are polynomials, the number of $\tau$-parameters does not depend on the degree of Tau-approximate solution. In the following, we briefly give the algorithm for implementing the proposed method as follows:
\begin{description}
  \item[Step 1] Enter $N$ and $\sigma=\frac{1}{\gamma}$.
  \item[Step 2.] Find the height vector $\mathbf{h}:=[h_{1},h_{2},...,h_{N}]$ from Theorem \ref{Int}.
  \item[Step 3.] Use Theorem \ref{newcan} to construct the sequence of vector canonical polynomials ${\bf{Q}}_{j}^{i}(t) $ and residuals ${\bf{R}}_{j}^{i}(t)$.
  \item[Step 4.] Fix the parameters $\tau^{N}_{j,i}$ from (\ref{javab}).
  \item[Step 5.] Use relation (\ref{wet}) to construct the Tau-approximate solution $Y_{N}(t)$.
\end{description}
\section{Convergence analysis}\label{sec3}
In this section,  the convergence property of the recursive Tau-approximation in $L_\infty$-norm is justified. we define $e_{\Pi_{N,\sigma}}u=u-\Pi_{N,\sigma}u$ as the M\"{u}ntz-Legendre truncation error and $e_{N,j}=y_j(t)-y_{N,j}(t)$ for $j=1,2,..., n$, as the error function of M\"{u}ntz-Legendre approximation of the solution $\{y_{j}(t)\}_{j=1}^n$ from the equation \eqref{1} by $\{y_{N,j}\}_{j=1}^n$.
\begin{theorem}\label{theorev1}
Assume that $D_t^k(g_i) \in L^2(\Omega)$, and $D^k(k_{ij}) \in L^2(\Omega \times \Omega)$ for $i,j=1,2,...,n,~k \ge 1$, then the recursive Tau approximations $\{y_{N,i}(t)\}_{i=1}^n$ converge to $\{y_{i}(t)\}_{i=1}^n$.
\end{theorem}
\begin{proof}
According the proposed approach, the approximate solutions $\{y_{N,j}\}_{j=1}^n$ satisfy  the following relation
\begin{equation}\label{erev11}
y_{N,i}(t)=H_{N,i}(t)+\Pi_{N,\sigma}(g_i)+\sum\limits_{j=1}^{n}{\int\limits_{\Omega}{(t-s)^{\alpha_{ij}-1} \Pi_{N,\sigma}(k_{ij}) y_{N,j}(s)ds}},~i=1,2,...,n.
\end{equation}

Subtracting \eqref{erev11} from \eqref{1}, and some simple manipulations we obtain
\begin{equation}\label{erev12}
e_{N,i}(t)=\psi_i-\sum\limits_{j=1}^{n}{\int\limits_{\Omega}{(t-s)^{\alpha_{ij}-1} k_{ij} e_{N,j}(s)ds}},~i=1,2,...,n,
\end{equation}
where
\[
\psi_i=-H_{N,i}(t)+e_{\Pi_{N,\sigma}}(g_i)-\sum\limits_{j=1}^{n}{\int\limits_{\Omega}{(t-s)^{\alpha_{ij}-1} e_{\Pi_{N,\sigma}}(k_{ij}) y_{N,j}(s)ds}},~i=1,2,...,n,
\]

Evidently, the relation \eqref{erev12} can be written as the following matrix formulation
\begin{equation}\label{erev13}
|E_{N}(t)| \le |\Psi(t)|+\|\tilde{K}\|_\infty \int\limits_{\Omega}{(t-s)^{\alpha-1} |E_{N}(s)|ds},
\end{equation}
where
\begin{equation*}
E_N(t) = [e_{N,1}(t), e_{N,2}(t),..., e_{N,n}(t)]^T,~~\Psi = [\psi_1(t), \psi_2(t),... \psi_n(t)]^T
\end{equation*}
and $|.|$ refers to componentwise absolute value. Applying Gronwall's inequality \cite{mokhs3} in the inequality \eqref{erev13} we can conclude
\[
\|E_N\|_\infty \le C \|\Psi\|_\infty,
\]
where $\|.\|_\infty$ stands for componentwise uniform norm. Consequently for $i=1,2,...,n$ we have
\begin{eqnarray}\label{erev14}
  \nonumber \|e_{N,i}\|_\infty & \le & \|\psi_i\|_\infty\\
 \nonumber  & \le &\|H_{N,i}\|_\infty +\|e_{\Pi_{N,\sigma}}(g_i)\|_\infty+ \sum\limits_{j=1}^{n}\|e_{\Pi_{N,\sigma}}(k_{ij})\|_\infty \|y_{N,j}\|_\infty\\
   & \le & \|H_{N,i}\|_\infty +\|e_{\Pi_{N,\sigma}}(g_i)\|_\infty+ \sum\limits_{j=1}^{n}  \|e_{\Pi_{N,\sigma}}(k_{ij})\|_\infty (\|y_j\|_\infty+\|e_{N,j}\|_\infty).
\end{eqnarray}

Applying Lemma \ref{lemrev3} in \eqref{erev14}, it can be deduced that for sufficiently large values of $N$ we have
\begin{equation}\label{erev15}
\|e_{N,j}\|_\infty \le  \|H_{N,i}\|_\infty.
\end{equation}
and thereby the proposed recursive Tau scheme converges iff $\|H_{N,i}\|_\infty$ as $N \to \infty$. To this end, assume that $X_{n}=\text{Span}\lbrace p_{0,\sigma}(t),...,p_{N,\sigma}(t)\rbrace$ and denote the orthogonal complement of the $X_{n}$ by $X^{\perp}_{n}$. Since,
\[
({H}_{N,i}(t), p_{j,\sigma}(t))=\left(\sum_{l=N+1}^{N+h_{i}}\tau^{N}_{l,i}p_{l,\sigma}(t),p_{j,\sigma}(t)\right) =0,\ \ \ i=0,1,...,n,\ j=0,1,...,N,
\]
then
\[
H_{N,i}(t) \in X^{\perp}_{n}.
\]

From Theorem 3 in \cite{Khajah}, we have the following results
\begin{enumerate}
\item[(A1)] The subspaces $X^{\perp}_{i},\ i=0,1,...$ form a decreasing sequence in $X=\text{Span}\lbrace p_{i,\alpha}(t): i=0,1,2,...\rbrace$, i.e.,
\[
\cdots \subset X^{\perp}_{3}\subset X^{\perp}_{2}\subset X^{\perp}_{1}.
\]
\item[(A2)]
\[
\lim_{n\rightarrow \infty}\mathrm{diam}(X^{\perp}_{n})=0;\ \ \ \mathrm{diam}(X^{\perp}_{n})=\underset{{a,b \in X^{\perp}_{n}}}{\sup}\Vert a-b \Vert.
\]
\item[(A3)]
\[
\overset{\infty}{\underset{i=1}{\bigcap}} X^{\perp}_{i}=\lbrace 0\rbrace.
\]
\end{enumerate}

From (A1), we obtain
\[
H_{N,i}(t)\in \overset{N}{\underset{l=0}{\bigcap}} X^{\perp}_{l},\ \ \ i=1,...,n,
\]
which along with (A2), yields
\begin{equation}\label{EEE}
\lim_{N \rightarrow \infty}H_{N,i}(t) \in \overset{\infty}{\underset{l=1}{\bigcap}} X^{\perp}_{l}=\lbrace 0\rbrace \ \ i.e.,  \ \  \lim_{N \rightarrow \infty}\Vert H_{N,i}(t)\Vert_{L^{2}}=0.
\end{equation}

Finally, inserting \eqref{EEE} into \eqref{erev15} completes the proof.
\end{proof}

\begin{remark}\label{Rem1}
The Parseval's identity yields
\[
\lim_{N \rightarrow \infty}\sum_{j=N+1}^{N+h_{i}}{\vert\tau^{N}_{j,i}\vert}^{2}=\lim_{N \rightarrow \infty}\Vert {H}_{N,i}(t)\Vert^{2}=0,\ \ \ \ i.e., \ \ \lim_{N \rightarrow \infty}{\vert\tau^{N}_{j,i}\vert}^{2}=0,~~i=1,2,...,n.
\]

Therefore, we conclude from \eqref{erev15} that the convergence rate of approximate solution is same as the convergence rate of $\tau^{N}_{j,i}$ to zero.
\end{remark}
\section{Numerical illustration}\label{sec4}
This section consists of the implementation of our method on some examples. The results obtained from our method show the accuracy and superiority of the method compared with those in \cite{Ha,hybrid}. All calculations are done in Maple 2018 software.
\begin{example}\label{3x}
Consider the problem \eqref{1rev} with
\[
K(t,s)=\left[
  \begin{array}{cccc}
   0& && \frac{(t-s)^{\alpha -1}}{\Gamma(\alpha)} \\
    \\
     -\frac{(t-s)^{\alpha -1}}{\Gamma(\alpha)}&&& -\frac{(t-s)^{\alpha -1}}{\Gamma(\alpha)} \\
  \end{array}
\right],
\]
\[
G(t)=\left[
  \begin{array}{c}
0\\\\
\frac{\Gamma(\alpha+1)t^{2\alpha+1}}{\Gamma(2\alpha+1)}+\frac{\Gamma(\alpha+1)\pi csc(\pi \alpha)t}{\Gamma(-\alpha-1)\Gamma(2-\alpha)}+\frac{\Gamma(\alpha+1)\pi csc(\pi \alpha)t^{\alpha+1}}{\Gamma(-\alpha-1)\Gamma(\alpha+2)} \\
  \end{array}
\right],\ \ \
\]
\[
 Y(t)=[t^{1+\alpha},\frac{\pi \alpha (\alpha+1)csc(\pi \alpha)}{\Gamma(1-\alpha)}t]^{T}.
\]
For $\alpha=\frac{1}{4}$, we have
\[g_1(t)=0, \quad
g_{2}(t)=\frac{\sqrt{2 \pi}}{3\Gamma(\frac{3}{4})}t^{\frac{6}{4}}+t^{\frac{5}{4}}+\frac{5\sqrt{2}\pi}{16\Gamma(\frac{3}{4})}t=\sum_{\ell=0}^{6}g_{\ell,2}t^{\frac{\ell}{4}}.
\]
From  \textbf{Step 2.}, we obtain
\[
\Lambda_{12}=-\Lambda_{21}=-\Lambda_{22}=\frac{1}{\Gamma(\frac{1}{4})}\left[
  \begin{array}{cccccccc}
  0&  \hspace{0.2cm} 4&0&\cdots\\
     0&\hspace{0.2cm} 0 &\beta(\frac{1}{4},\frac{5}{4})& \ddots\\
     0&\hspace{0.2cm} 0 &0&  \ddots \\
     \vdots &\hspace{0.2cm} \vdots &\vdots &\ddots  \\
  \end{array}
\right]
,\ \ \ \Lambda_{11}=\mathbf{0}.
\]
\[
\left\{
  \begin{array}{ll}
      h_{11}=0,\ h_{12}=1, \ h_{1}=\max \lbrace 0,1 \rbrace=1,\\
   h_{21}=1,\ h_{22}=1,  \ h_{2}=\max \lbrace 1,1 \rbrace=1,\\
     \Delta_{1}=\Delta_{2}=0,\\
      \mathbf{h}=[h_{1},h_{2}]=[1,1]\\
  \end{array}
\right.
\]
From  \textbf{Step 3.}, the fractional vector canonical polynomials ${\bf{Q}}_{1}^{r}(t), {\bf{Q}}_{2}^{r}(t)$ and the related residuals ${\bf{R}}_{1}^{r}(t), {\bf{R}}_{2}^{r}(t)$ are determined as
\[
{\bf{Q}}_{1}^{0}(t)={\bf{Q}}_{2}^{0}(t)=[0,0]^{T},\ \ \ \ \ \ \ {\bf{R}}_{1}^{0}=[-1,0]^{T},\ \ {\bf{R}}_{2}^{0}=[0,-1]^{T},
\]
\[
\left\{
  \begin{array}{ll}
{\bf{Q}}_{1}^{r+1}(t)=\displaystyle\sum_{i=1}^{2}d_{i1}\left(t^{\frac{r}{4}}{\bf{e}}_{i}-
\sum_{v=1}^{2} \displaystyle\sum_{l=1}^{r+1}\widetilde{\Lambda}_{vi}(r+1,l){\bf{Q}}_{v}^{(l-1)}(t)
\right),\\
{\bf{R}}_{1}^{r+1}=-\displaystyle\sum_{i=1}^{2}d_{i1}\left(\sum_{v=1}^{2} \displaystyle\sum_{l=1}^{r+1}\widetilde{\Lambda}_{vi}(r+1,l)
{\bf{R}}_{v}^{(l-1)}\right),
\end{array}
\right.
\]
and
\[
\left\{
  \begin{array}{ll}
{\bf{Q}}_{2}^{r+1}(t)=\displaystyle\sum_{i=1}^{2}d_{i2}\left(t^{\frac{r}{4}}{\bf{e}}_{i}-
\sum_{v=1}^{2} \displaystyle\sum_{l=1}^{r+1}\widetilde{\Lambda}_{vi}(r+1,l){\bf{Q}}_{v}^{(l-1)}(t)
\right),\\
 {\bf{R}}_{2}^{r+1}=-\displaystyle\sum_{i=1}^{2}d_{i2}\left(\sum_{v=1}^{2} \displaystyle\sum_{l=1}^{r+1}\widetilde{\Lambda}_{vi}(r+1,l)
{\bf{R}}_{v}^{(l-1)}\right),
\end{array}
\right.
\]
for all $r\geq 0$, in which
\[
 \widetilde{\Lambda}_{11}={\bf{I}}, \ \ \  \widetilde{\Lambda}_{22}={\bf{I}}-{\Lambda}_{22}=\left[
  \begin{array}{ccccccccc}
  1&\frac{4}{\Gamma(\frac{1}{4})}&0&0&\cdots\\
     0&1 &\frac{\beta(\frac{1}{4},\frac{5}{4})}{\Gamma(\frac{1}{4})}&0& \cdots\\
     0&\cdots &1& \frac{\beta(\frac{1}{4},\frac{3}{2})}{\Gamma(\frac{1}{4})}& \ddots \\
     \vdots &\cdots &\vdots &\ddots &\ddots  \\
  \end{array}
\right],
\]
\[
 \widetilde{\Lambda}_{12}=-{\Lambda}_{12}=\left[
  \begin{array}{ccccccccc}
  0&-\frac{4}{\Gamma(\frac{1}{4})}&0&0&\cdots\\
     0&0 &-\frac{\beta(\frac{1}{4},\frac{5}{4})}{\Gamma(\frac{1}{4})}&0& \cdots\\
     0&\cdots &0& -\frac{\beta(\frac{1}{4},\frac{3}{2})}{\Gamma(\frac{1}{4})}& \ddots \\
     \vdots &\cdots &\vdots &\ddots &\ddots  \\
  \end{array}
\right],
\]
\[
 \widetilde{\Lambda}_{21}=-{\Lambda}_{21}=\left[
  \begin{array}{ccccccccc}
  0&\frac{4}{\Gamma(\frac{1}{4})}&0&0&\cdots\\
     0&0 &\frac{\beta(\frac{1}{4},\frac{5}{4})}{\Gamma(\frac{1}{4})}&0& \cdots\\
     0&\cdots &0& \frac{\beta(\frac{1}{4},\frac{3}{2})}{\Gamma(\frac{1}{4})}& \ddots \\
     \vdots &\cdots &\vdots &\ddots &\ddots  \\
  \end{array}
\right],
\]
\[
P_{r}=\left[
  \begin{array}{cc}
    \widetilde{\Lambda}_{11}(r+1,r+2)  & \widetilde{\Lambda}_{12}(r+1,r+2) \\
      \widetilde{\Lambda}_{21}(r+1,r+2) & \widetilde{\Lambda}_{22}(r+1,r+2) \\
  \end{array}
\right]=\left[
  \begin{array}{cc}
    0 &  -\frac{\beta(\frac{1}{4},\frac{r+4}{4})}{\Gamma(\frac{1}{4})} \\
    \frac{\beta(\frac{1}{4},\frac{r+4}{4})}{\Gamma(\frac{1}{4})}&\frac{\beta(\frac{1}{4},\frac{r+4}{4})}{\Gamma(\frac{1}{4})}  \\
  \end{array}
\right],
\]
\[
D={P_{r}}^{-1}=\left[
  \begin{array}{cc}
    d_{11}  & d_{12}\\
      d_{21} & d_{22} \\
  \end{array}
\right]=\frac{1}{\Gamma(\frac{1}{4})} \left[
  \begin{array}{cc}
    \frac{1}{\beta(\frac{1}{4},\frac{r+4}{4})} &  \frac{1}{\beta(\frac{1}{4},\frac{r+4}{4})}\\
    \frac{-1}{\beta(\frac{1}{4},\frac{r+4}{4})}&0  \\
  \end{array}
\right].
\]

 From \textbf{Step 4.}, the unknown $\tau$-parameters  $\tau^{N}_{N+1,1}$, $\tau^{N}_{N+1,2}$ are determined by solving   $2\times 2$ linear algebraic system
\begin{equation}\label{2dr2}
M_{2\times 2}\underline{\tau}=\underline{b}_{2\times 1},
\end{equation}
where
\[
M_{2\times 2}=\left[\sum\limits_{\ell=0}^{N+1}c_{N+1,\ell}{\bf{R}}_{1}^{\ell};\sum\limits_{\ell=0}^{N+1}c_{N+1,\ell}{\bf{R}}_{2}^{\ell}\right],\  \ \underline{b}_{2\times 1}=-\sum_{\ell=0}^{6}g_{\ell,1}{\bf{R}}_{1}^{\ell}-\sum_{\ell=0}^{6}g_{\ell,2}{\bf{R}}_{2}^{\ell},
\]
\[
\underline{\tau}=[\tau^{N}_{N+1,1}, \tau^{N}_{N+1,2}]^{T},
\]
therefore, for $N=6$ we get from (\ref{2dr2})
\[
\underline{\tau}=[0,0]
\]
and so from \textbf{Step 5.}, the Tau-approximate solution is
\begin{eqnarray}
Y_{N}(t)&=\sum_{\ell=0}^{6}g_{\ell,2}{\bf{Q}}_{2}^{\ell}(t)+
\tau^{6}_{7,1}\sum\limits_{\ell=0}^{7}c_{7,\ell}{\bf{Q}}_{1}^{\ell}(t)+\tau^{6}_{7,2}\sum\limits_{\ell=0}^{7}c_{7,\ell}{\bf{Q}}_{2}^{\ell}(t)=[t^{\frac{5}{4}},\frac{5\sqrt{2}\pi}{16\Gamma(\frac{3}{4})}t]^{T}.
\end{eqnarray}
Therefore, the exact solution of the problem is obtained. The numerical results of hybrid numerical method \cite{hybrid} are shown in Table \ref{x3} in which $m$ is the number of subintervals of $\Omega$.

\begin{table}[!h]
\centering
\caption{The results of example  \ref{3x} in \cite{hybrid}.}
\begin{tabular}{llllllll}
\hline\noalign{\smallskip}
m ~~& &Error ($\alpha=1/4$)&m&&Error ($\alpha=2/3$)\\
\noalign{\smallskip}\hline\noalign{\smallskip}
15& &2.11$\times 10^{-6}$&32&&5.53$\times 10^{-8}$\\
 \noalign{\smallskip}\hline\noalign{\smallskip}
34& &1.45$\times 10^{-7}$&71&&3.58$\times10^{-9}$\\
\noalign{\smallskip}\hline\noalign{\smallskip}
75& &9.43$\times 10^{-9}$&155&&2.28$\times10^{-10}$ \\
 \noalign{\smallskip}\hline\noalign{\smallskip}
166& &6.00$\times 10^{-10}$&331&&1.44$\times10^{-11}$ \\
\noalign{\smallskip}\hline
\end{tabular}\label{x3}
\end{table}
\begin{table}
 \center
\caption{\small{The numerical results of Ref. \cite{Ha}} for example \ref{opv}}
\label{Tab2}
\begin{tabular}{llllll}
\hline\noalign{\smallskip}
t & $e_{1,50}(t)$& $e_{2,50}(t)$& $e_{1,100}(t)$& $e_{2,100}(t)$\\
\noalign{\smallskip}\hline\noalign{\smallskip}
0.1&5.5593$\times 10^{-5}$&4.5601$\times 10^{-4}$&3.1699$\times 10^{-5}$&2.4869$\times 10^{-4}$\\
0.2&4.6189$\times 10^{-4}$&1.8036$\times 10^{-3}$&2.4333$\times 10^{-4}$&9.3104$\times 10^{-4}$\\
0.3&1.4361$\times 10^{-3}$&3.6764$\times 10^{-3}$&7.3706$\times 10^{-4}$&1.8644$\times 10^{-3}$\\
0.4&2.9937$\times 10^{-3}$&5.6995$\times 10^{-3}$&1.5166$\times 10^{-3}$&2.8652$\times 10^{-3}$\\
0.5&4.9480$\times 10^{-3}$&7.4979$\times 10^{-3}$&2.4872$\times 10^{-3}$&3.7497$\times 10^{-3}$\\
0.6&6.9104$\times 10^{-3}$&8.6969$\times 10^{-3}$&3.4558$\times 10^{-3}$&4.3343$\times 10^{-3}$\\
0.7&8.2904$\times 10^{-3}$&8.9215$\times 10^{-3}$&4.1307$\times 10^{-3}$&4.4353$\times 10^{-3}$\\
0.8&8.2954$\times 10^{-3}$&7.7968$\times 10^{-3}$&4.1217$\times 10^{-3}$&3.8690$\times 10^{-3}$\\
0.9&5.9308$\times 10^{-3}$&4.9479$\times 10^{-3}$&2.9405$\times 10^{-3}$&2.4518$\times 10^{-3}$\\
1&0&0&0&0\\
\noalign{\smallskip}\hline
\end{tabular}
\end{table}
\end{example}
\begin{example}\label{opv}\cite{Ha}
Consider the problem \eqref{1rev} with
\[
K(t,s)=\left[
  \begin{array}{cccc}
    (t-s)^{-1/5}& && (t-s)^{-2/5} \\
    \\
     (t-s)^{-3/5}&&& (t-s)^{-4/5} \\
  \end{array}
\right],
\]
\[
G(t)=\left[
  \begin{array}{c}
  t+t^{2}-\frac{25}{6552}t^{\frac{8}{5}}(130t^{\frac{6}{5}}+182t^{\frac{1}{5}}-210t+273) \\\\
 t-t^{2}-\frac{25}{924}t^{\frac{6}{5}}(55t^{\frac{6}{5}}+66t^{\frac{1}{5}}-140t+154) \\
  \end{array}
\right],
\]
\[
 Y(t)=[t+t^2,t-t^2]^{T}.
\]
\end{example}

During the implementation of the our method, similar to the previous example, the Tau-solution coincide with exact solution. Table \ref{Tab2} represent the results reported in \cite{Ha} for different $n$.
\begin{figure}
\centering
\begin{minipage}{4cm}
\parbox{4cm}{
\includegraphics[width=5cm]{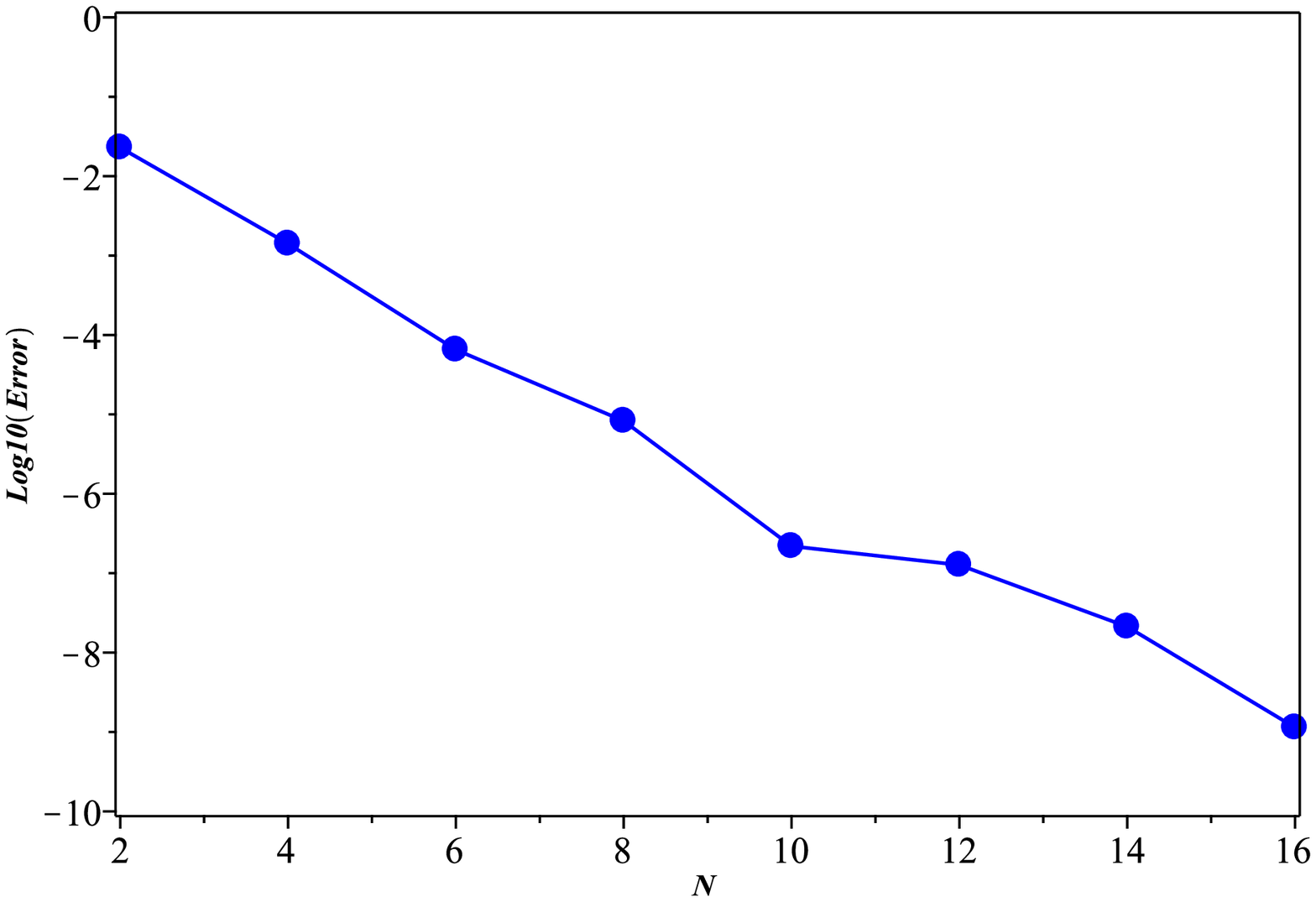}
\caption{\small{Error of $y_{1}(t)$ of Example \ref{vc}  for various $N$.}}
\label{FE1}}
\end{minipage}
\qquad \qquad \qquad
\begin{minipage}{4cm}
\includegraphics[width=5cm]{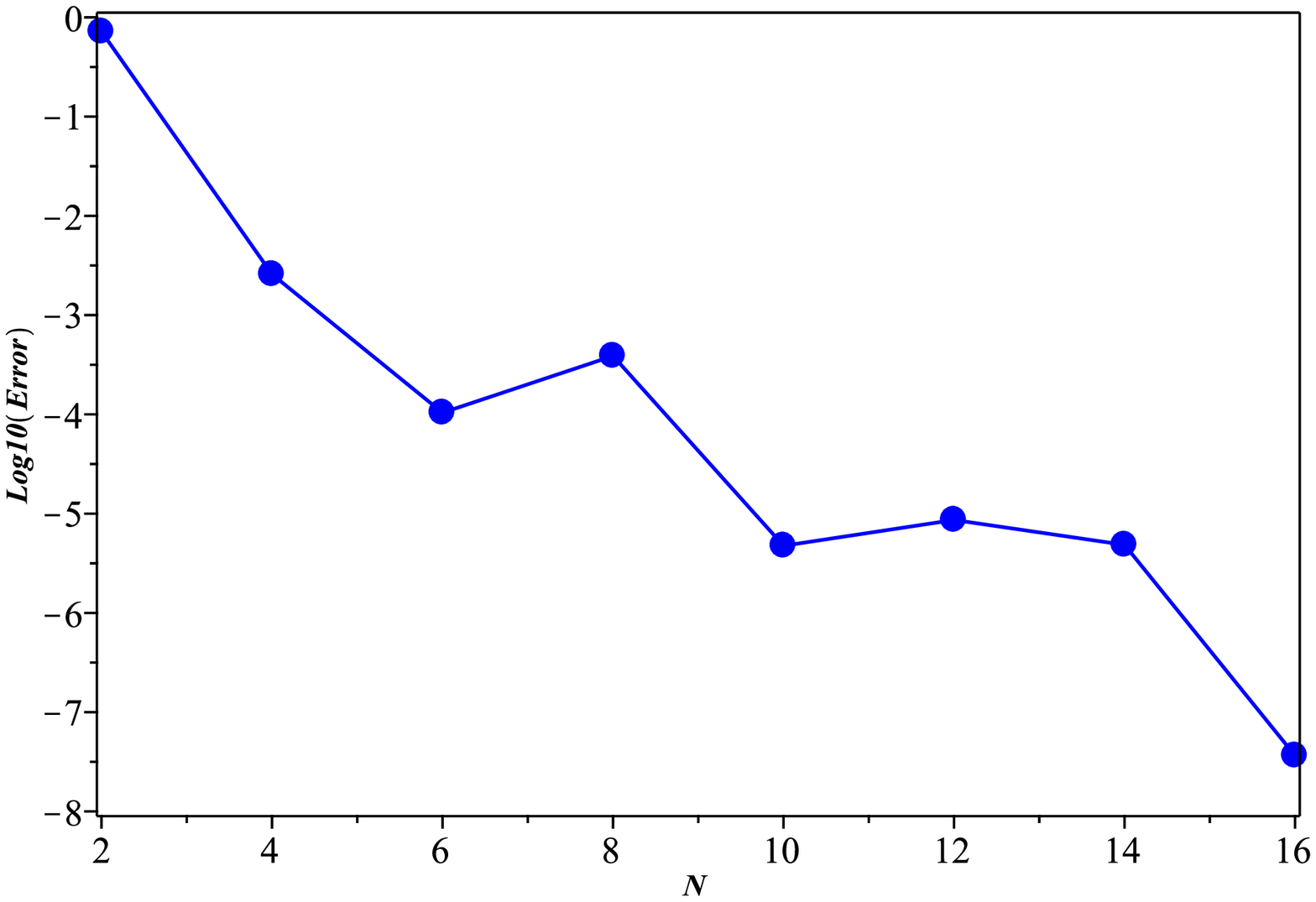}
\caption{\small{Error of $y_{2}(t)$ of Example \ref{vc} for various $N$.}}
\label{FE2}
\end{minipage}
\end{figure}
\begin{figure}
\centering
\begin{minipage}{4cm}
\parbox{4cm}{
\includegraphics[width=5cm]{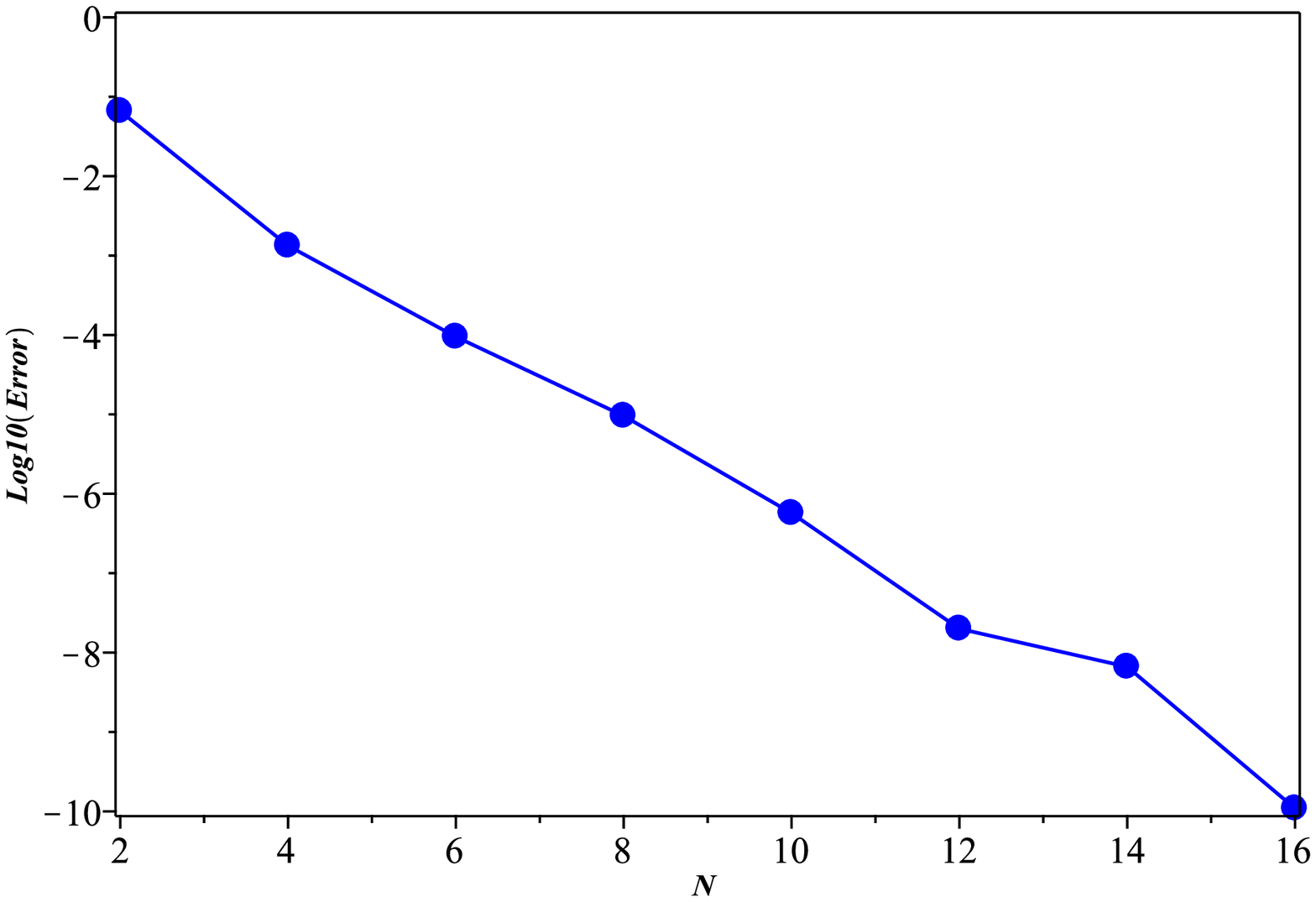}
\caption{\small{Behavior of the $\Vert \widehat{\tau}_{1,N}\Vert$ for the Example \ref{vc} for various $N$.}}
\label{fig6T}}
\end{minipage}
\qquad \qquad \qquad
\begin{minipage}{4cm}
\includegraphics[width=5cm]{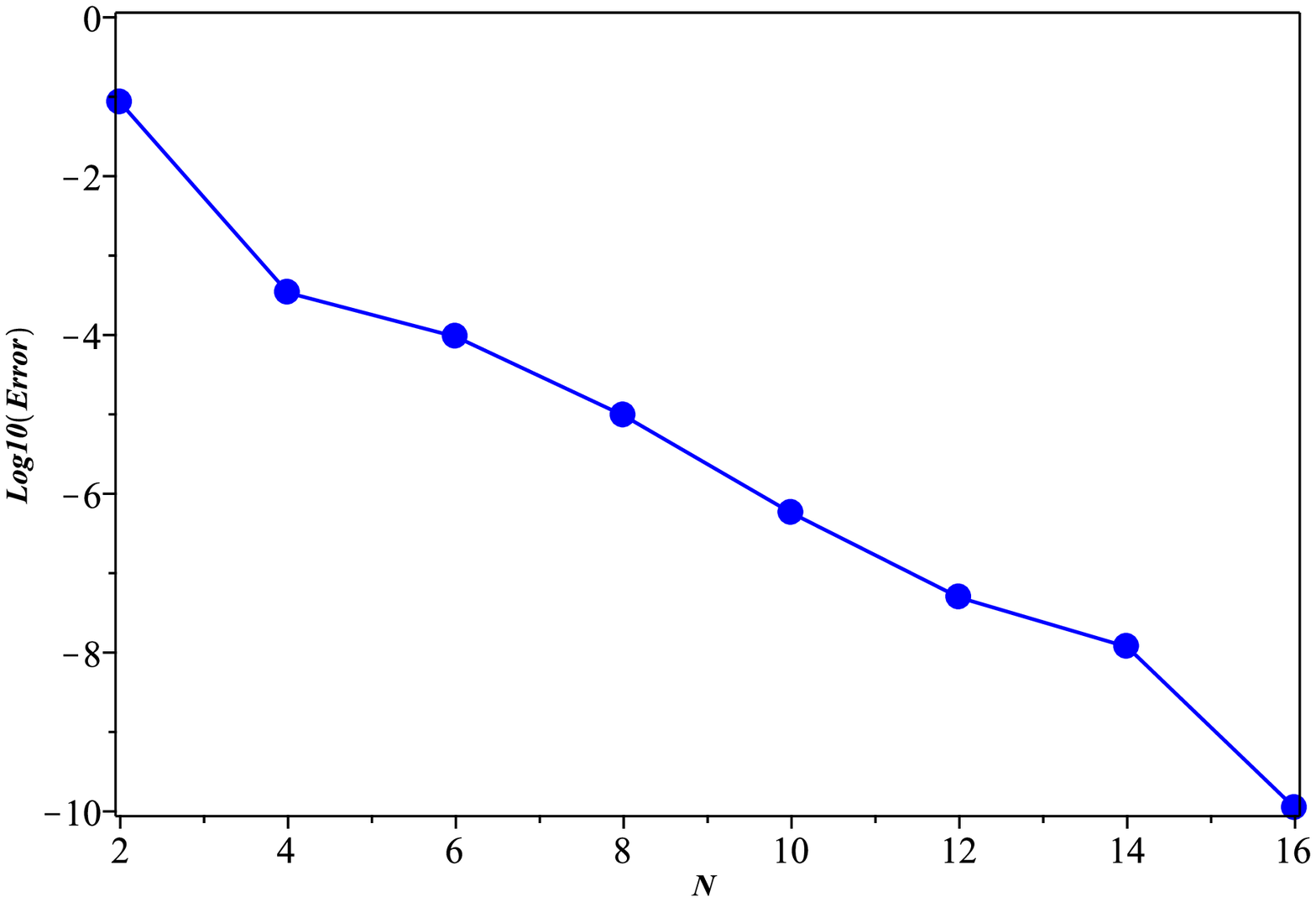}
\caption{\small{Behavior of the $\Vert \widehat{\tau}_{2,N}\Vert$ for the Example \ref{vc} for various $N$.}}
\label{fig7T}
\end{minipage}
\end{figure}
\begin{example}\label{vc}
Consider the problem \eqref{1rev} with
\begin{equation*}
K(t,s)=\left[
  \begin{array}{cccc}
    (t-s)^{-3/4}&&(t-s)^{-1/4}\\
 (t-s)^{-3/4}&&(t-s)^{-2/4} \\
  \end{array}
\right],\ \ \
 Y(t)=\left[
 \begin{array}{c}
 \arctan (\sqrt{t})\\
 \sqrt[4]{t^3}
 \end{array}
 \right].
\end{equation*}
\end{example}
From  \textbf{Step 2.}, we have
\[
\left\{
  \begin{array}{ll}
      h_{11}=1,\ h_{12}=3, \ h_{1}=\max \lbrace 1,3 \rbrace=3,\\
   h_{21}=1,\ h_{22}=2,  \ h_{2}=\max \lbrace 1,2 \rbrace=2,\\
     \Delta_{1}=1,\ \Delta_{2}=0,\\
      \mathbf{h}=[h_{1},h_{2}]=[3,2].
  \end{array}
\right.
\]
Define
\[
\Vert \widehat{\boldsymbol{\tau}}_{i,N}\Vert:=\underset{1\leq j \leq h_{i}}{\max} \vert \tau^{N}_{N+j,i}\vert,\ \ \ i=1,2.
\]
The obtained errors and Tau parameters behavior are reported in Table \ref{Tab1}. From Figs. \ref{FE1}, \ref{FE2}, it is obvious that as $N$ increases,  the $L_\infty$-norm of errors is decayed.
Moreover, the obtained theoretical prediction in Remark \ref{Rem1} is confirmed by Figs. \ref{fig6T}, \ref{fig7T}, i.e., the approximation errors reduce same as reduction of the Tau-parameters to zero.
\begin{table}
 \center
\caption{\small{The numerical results of Example \ref{vc}} }
\label{Tab1}
\begin{tabular}{lllllllll}
\hline\noalign{\smallskip}
N ~~&4&8&10&12&14&16&18&20\\
\noalign{\smallskip}\hline\noalign{\smallskip}
 $ \Vert e_{1,N}\Vert_{\infty}$ &1.41e-03&8.27e-06 &2.19e-07&1.27e-07&2.12e-08 &1.14e-09&1.26e-10&7.38e-12 \\
 \noalign{\smallskip}\hline\noalign{\smallskip}
 $ \Vert e_{2,N}\Vert_{\infty}$&2.58e-03&3.88e-04&4.75e-06&8.64e-06&4.85e-06&3.65e-08&9.19e-09&2.25e-10 \\
  \noalign{\smallskip}\hline\noalign{\smallskip}
 $\Vert \widehat{\tau}_{1,N}\Vert$&1.33e-03&9.57e-06&5.71e-07 &2.00e-08 &6.63e-09&1.09e-10 & 2.51e-11&2.75e-12 \\
  \noalign{\smallskip}\hline\noalign{\smallskip}
 $\Vert \widehat{\boldsymbol{\tau}}_{2,N}\Vert$&3.40e-04&9.66e-06 &5.73e-07&4.94e-08&1.18e-08&1.10e-10 &2.54e-11&2.76e-12 \\
\noalign{\smallskip}\hline
\end{tabular}
\end{table}
\begin{table}[!h]
\centering
\caption{The results of Example  \ref{0111}.}
\begin{tabular}{llllllll}
\hline\noalign{\smallskip}
N ~~& 2&4&6&8&10&12&14\\
\noalign{\smallskip}\hline\noalign{\smallskip}
 $ \Vert e_{1,N}\Vert_{\infty}$&3.06e-2& 1.17e-3&2.06e-4& 1.19e-6&3.89e-7&3.51e-9&2.85e-10 \\
 \noalign{\smallskip}\hline\noalign{\smallskip}
 $ \Vert e_{2,N}\Vert_{\infty}$&5.08e-3&3.24e-3 & 5.41e-4& 9.43e-6 & 2.19e-7 & 8.72e-9& 3.06e-12 \\
\noalign{\smallskip}\hline\noalign{\smallskip}
 $\Vert \widehat{\tau}_{1,N}\Vert$&7.64e-3&1.95e-4&2.58e-4&1.19e-7&3.24-8&2.51e-10 &1.78e-11 \\
 \noalign{\smallskip}\hline\noalign{\smallskip}
$\Vert \widehat{\tau}_{2,N}\Vert$&1.27e-3& 5.41e-4&6.76e-6&9.43e-7&1.83e-8&6.02e-10&1.79e-11  \\
\noalign{\smallskip}\hline
\end{tabular}\label{147}
\end{table}

 \begin{figure}
\centering
\begin{minipage}{4cm}
\parbox{4cm}{
\includegraphics[width=5cm]{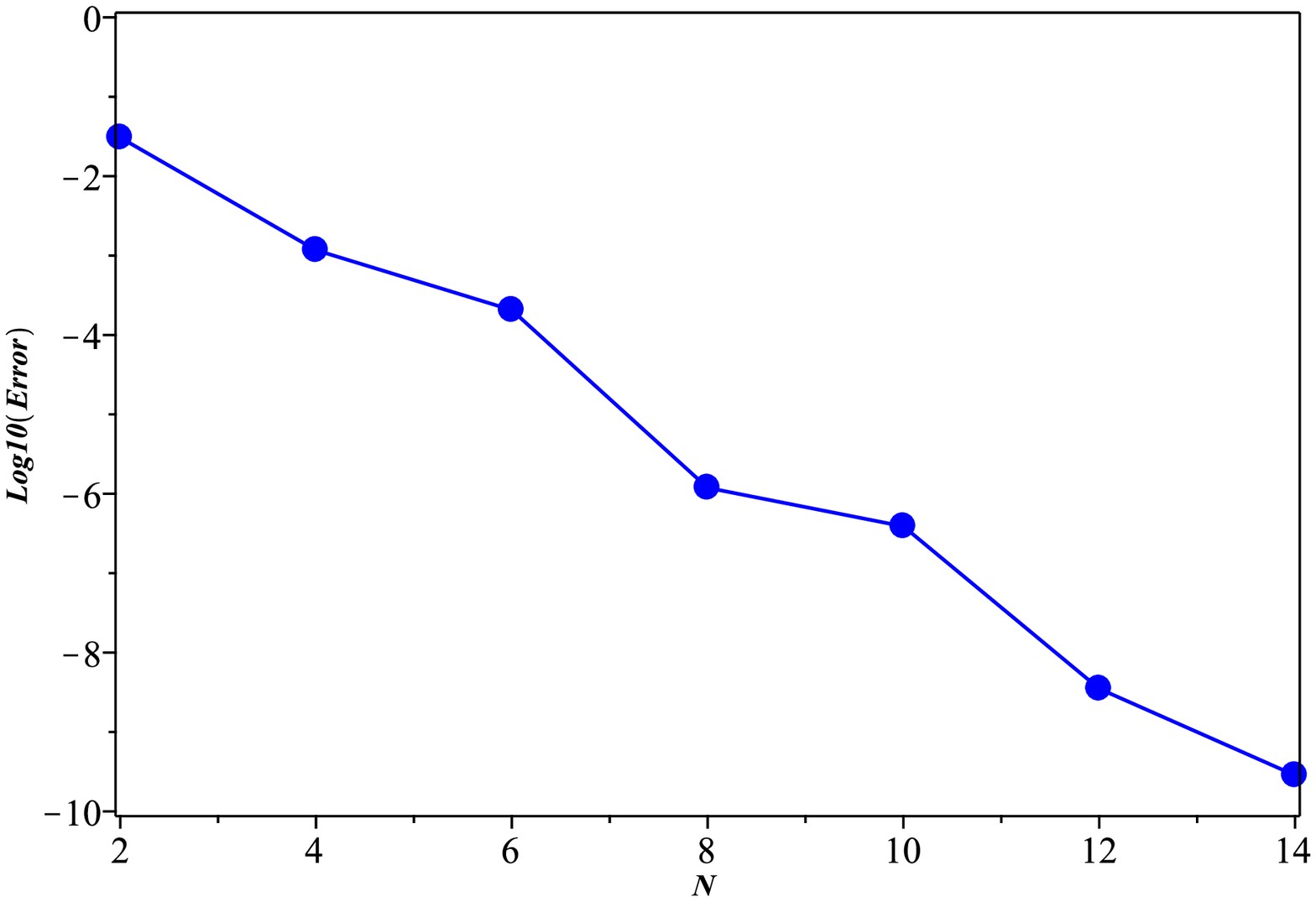}
\caption{\small{Error of $y_{1}(t)$ of Example \ref{0111}  for various $N$.}}
\label{fig4}}
\end{minipage}
\qquad \qquad \qquad
\begin{minipage}{4cm}
\includegraphics[width=5cm]{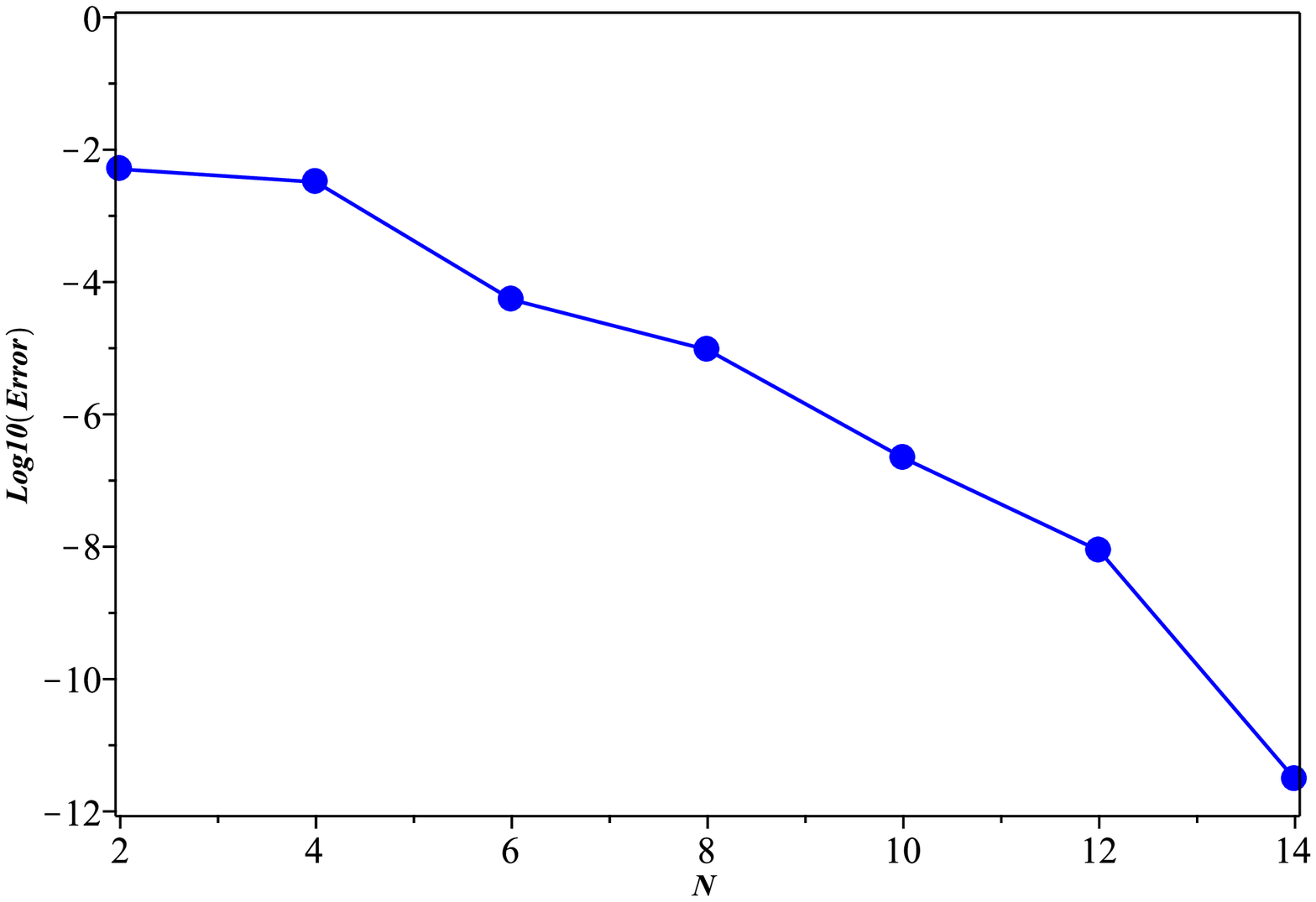}
\caption{\small{Error of $y_{2}(t)$ of Example \ref{0111} for various $N$.}}
\label{fig5}
\end{minipage}
\end{figure}
\begin{figure}
\centering
\begin{minipage}{4cm}
\parbox{4cm}{
\includegraphics[width=5cm]{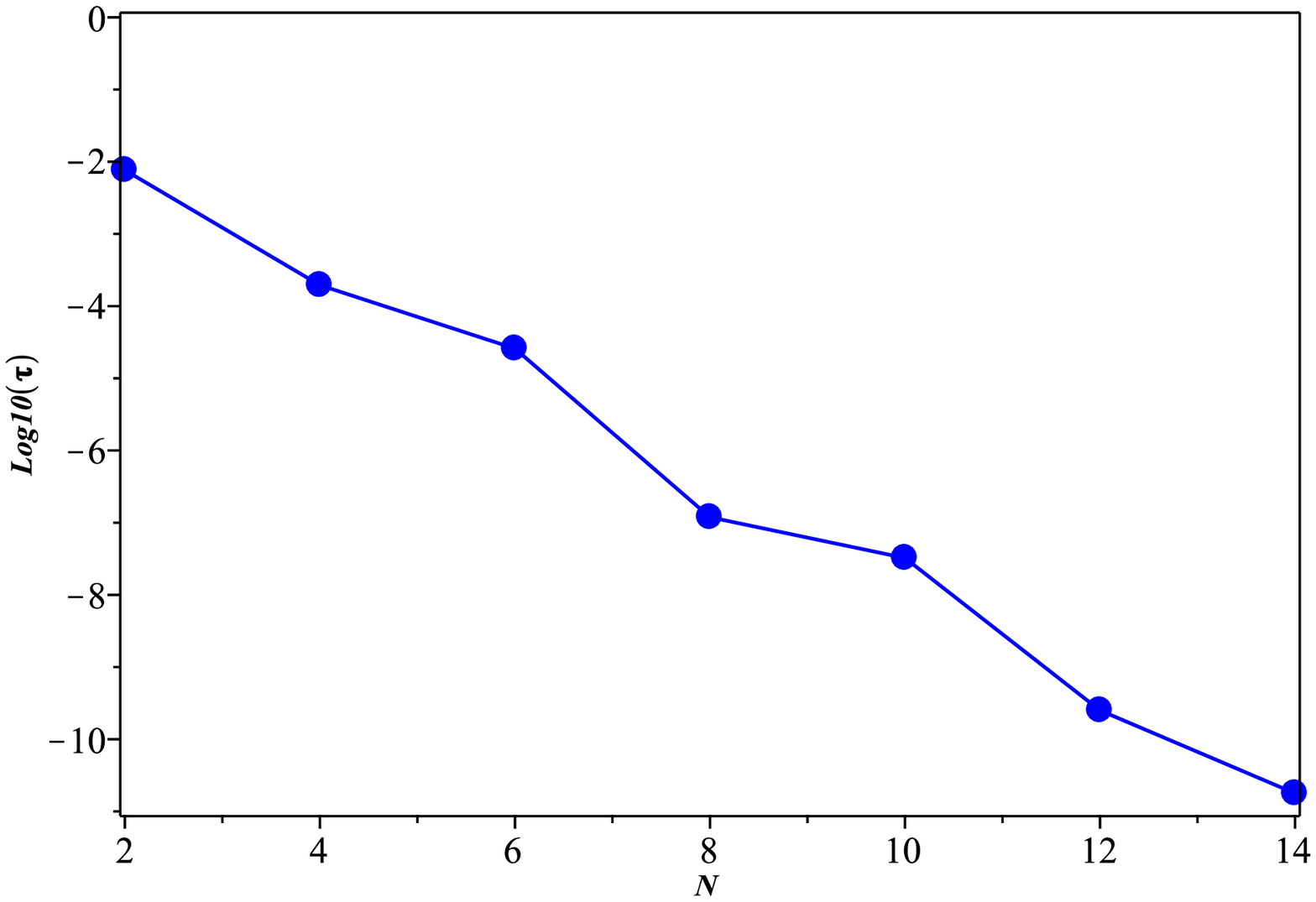}
\caption{\small{Behavior of the $\Vert \widehat{\tau}_{1,N}\Vert$ for the Example \ref{0111} for various $N$.}}
\label{fig6}}
\end{minipage}
\qquad \qquad \qquad
\begin{minipage}{4cm}
\includegraphics[width=5cm]{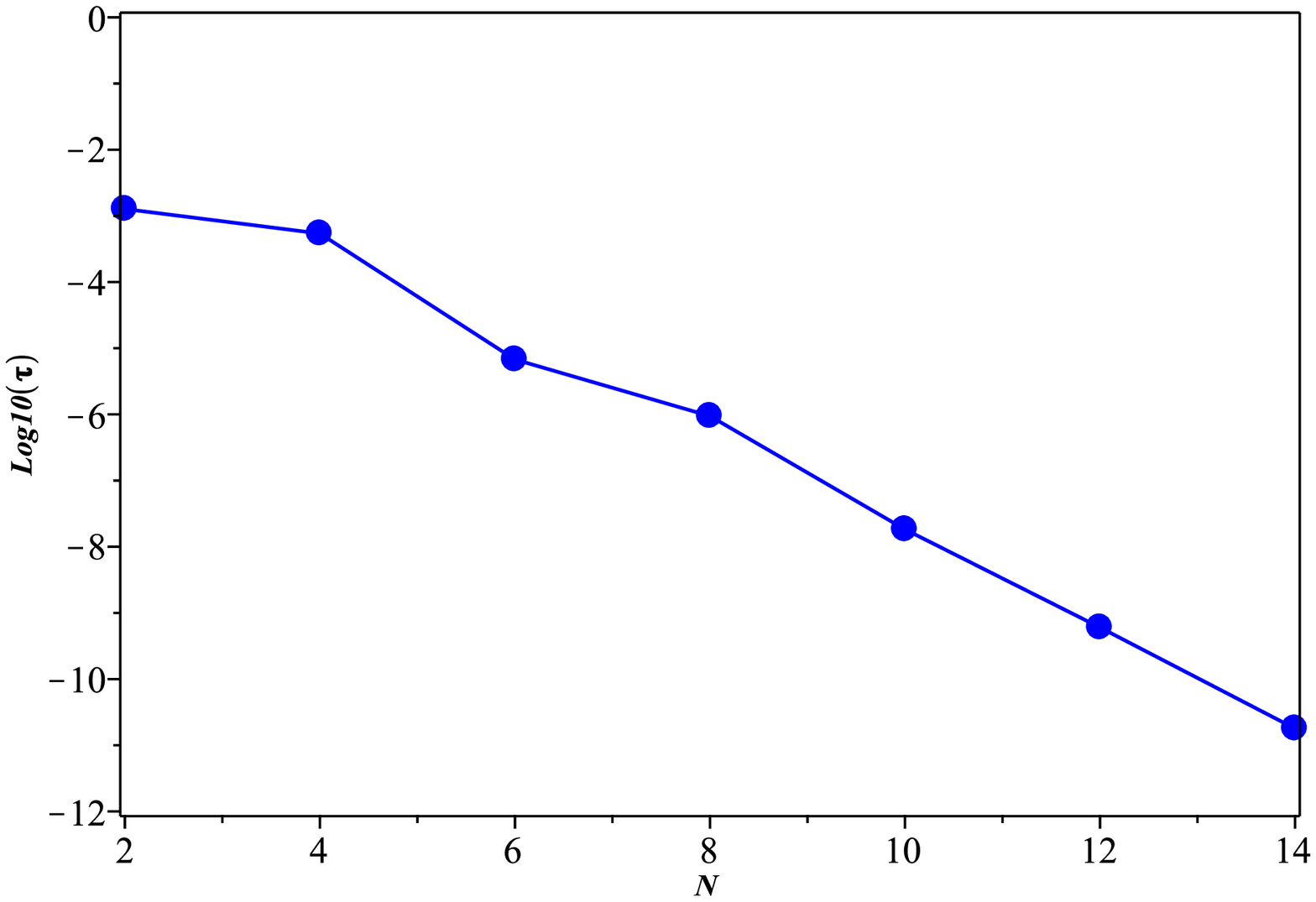}
\caption{\small{Behavior of the $\Vert \widehat{\tau}_{2,N}\Vert$ for the Example \ref{0111} for various $N$.}}
\label{fig7}
\end{minipage}
\end{figure}
\begin{example}\label{0111}
Consider the problem \eqref{1rev} with
\[
K(t,s)=\left[
  \begin{array}{cccc}
    -(t-s)^{-1/2}& && -(t-s)^{-1/2} \\
    \\
     (t-s)^{-1/2}&&& 0 \\
  \end{array}
\right],\ \ \
G(t)=\left[
  \begin{array}{c}
 2\sqrt{t}+\frac{\pi}{2}t\\\\
1-\sqrt{t}-e^{\pi t}erfc(\sqrt{\pi t}) \\
  \end{array}
\right],\ \ \
\]
\[
 Y(t)=\left[
 \begin{array}{c}
 1-e^{\pi t}erfc(\sqrt{\pi t})\\
 \sqrt{t}
 \end{array}
\right],
\]
in which $erfc(t):=1-\frac{2}{\sqrt{t}}\int_{0}^{t}e^{-t^{2}}dt$ is the complementary error function.
\end{example}
We have
\[
\left\{
  \begin{array}{ll}
      h_{11}=1,\ h_{12}=1, \ h_{1}=\max \lbrace 1,1 \rbrace=0,\\
   h_{21}=1,\ h_{22}=0,  \ h_{2}=\max \lbrace 1,0 \rbrace=1,\\
     \Delta_{1}=\Delta_{2}=0,\\
      \mathbf{h}=[h_{1},h_{2}]=[1,1].
  \end{array}
\right.
\]
The errors and behavior of Tau parameters for various values of  $N$ are reported in Table \ref{147} and Figs. \ref{fig4}-\ref{fig7}.
The subject of Remark \ref{Rem1} is  also conﬁrmed by these results.

\section{Conclusion}\label{sec5}
In this paper, we developed the recursive approach of the spectral Tau method (recursive Tau method) for solving a class of Abel-Volterra integral equations system based on a new set of fractional vector polynomials basis. These types of equations have a singularity at $t=0$, which indicates the non-smoothness of the solution. To recover the order of convergence of the recursive Tau method, we constructed a set of vector canonical polynomials of fractional order as basis functions by means of a recursive algorithm. This allows us to construct an accurate Tau-approximate solution. The numerical results confirmed the accuracy of the method as an exponential rate of convergence of it.  In the next research, we intend to implement our method to solve the system of Abel-Volterra integro-differential equations.

\end{document}